\documentclass{amsart}

\usepackage{amssymb}
\usepackage[pdftex]{graphicx,color}

\usepackage{url}

\usepackage{tikz}
\usepackage{tkz-graph}
\usetikzlibrary{arrows}
\pgfarrowsdeclare{mytip}{mytip}{%
  \setlength{\arrowsize}{1pt}  
  \addtolength{\arrowsize}{.5\pgflinewidth}  
  \pgfarrowsrightextend{0}  
  \pgfarrowsleftextend{-5\arrowsize}  
}{%
  \setlength{\arrowsize}{1pt}  
  \addtolength{\arrowsize}{.5\pgflinewidth}  
  \pgfpathmoveto{\pgfpoint{-5\arrowsize}{1.5\arrowsize}}  
  \pgfpathlineto{\pgfpointorigin}  
  \pgfpathlineto{\pgfpoint{-5\arrowsize}{-1.5\arrowsize}}  
  \pgfusepathqstroke  
}

\copyrightinfo{2012}{David Perkinson}

\newtheorem{theorem}{Theorem}[section]
\newtheorem{lemma}[theorem]{Lemma}
\newtheorem{prop}[theorem]{Proposition}
\newtheorem{thm}[theorem]{Theorem}
\newtheorem{cor}[theorem]{Corollary}

\theoremstyle{definition}

\newtheorem{Example}[theorem]{Example}
\newenvironment{example}
  {\begin{Example}\rm}{\hfill$\Box$\end{Example}}
\newtheorem{notation}[theorem]{Notation}

\theoremstyle{remark}
\newtheorem{remark}[theorem]{Remark}

\numberwithin{equation}{section}

\newcommand{\N}{\mathbb{N}}
\newcommand{\Z}{\mathbb{Z}}

\newcommand{\Q}{\mathbb{Q}}

\newcommand{\tV}{\widetilde{V}}
\newcommand{\tG}{\mathcal{H}(\Gamma)}
\newcommand{\cmax}{c_{\mathrm{max}}}

\newcommand{\sg}{\mathrm{S}\Gamma}
\newcommand{\tg}{\text{M}\Gamma}
\newcommand{\mg}{\Gamma^{\mathrm{mob}}}
\DeclareMathOperator{\wt}{wt}
\DeclareMathOperator{\outdeg}{outdeg}

\newcommand{\tD}{\widetilde{\Delta}}
\newcommand{\tL}{\widetilde{\mathcal{L}}}
\newcommand{\sand}{\mathcal{S}}
\newcommand{\po}{\mathcal{O}}

\newcommand{\ba}{\begin{eqnarray*}}
\newcommand{\ea}{\end{eqnarray*}}
\newcommand{\be}{\begin{enumerate}}
\newcommand{\ee}{\end{enumerate}}

\begin{document}

\title[Sandpiles and Dominos]{Sandpiles and Dominos}


\author[Florescu]{Laura Florescu}
\address{Courant Institute, NYU, New York}
\email{florescu@cims.nyu.edu}

\author[Morar]{Daniela Morar}
\address{Department of Economics, University of Michigan, Ann Arbor}
\email{morard@umich.edu}

\author[Perkinson]{David Perkinson}
\address{Department of Mathematics, Reed College}
\email{davidp@reed.edu}

\author[Salter]{Nick Salter}
\address{Department of Mathematics, University of Chicago}
\email{nks@math.uchicago.edu}

\author[Xu]{Tianyuan Xu}
\address{Department of Mathematics, University of Oregon}
\email{tianyuan@uoregon.edu}

\date{\today}

\begin{abstract}
  We consider the subgroup of the abelian sandpile group of the grid graph
  consisting of configurations of sand that are symmetric with respect to
  central vertical and horizontal axes.  We show that the size of this group
  is (i) the number of domino tilings of a corresponding weighted rectangular
  checkerboard; (ii) a product of special values of Chebyshev polynomials; and
  (iii) a double-product whose factors are sums of squares of values of
  trigonometric functions.  We provide a new derivation of the formula due to
  Kasteleyn and to Temperley and Fisher for counting the number of domino tilings of
  a $2m\times 2n$ rectangular checkerboard and a new way of counting the number
  of domino tilings of a $2m\times 2n$ checkerboard on a M\"obius strip.
\end{abstract}

\maketitle

\section{Introduction}\label{section:Introduction}
This paper relates the Abelian Sandpile Model (ASM) on a grid graph to domino
tilings of checkerboards.  The ASM is, roughly, a game in which one places
grains of sand on the vertices of a graph, $\Gamma$, whose vertices and edges we
assume to be finite in number.  If the amount of sand on a vertex reaches a
certain threshold, the vertex becomes unstable and fires, sending a grain of
sand to each of its neighbors.  Some of these neighbors, in turn, may now be
unstable.  Thus, adding a grain of sand to the system may set off a cascade of
vertex firings.  The resulting ``avalanche'' eventually subsides, even
though our graph is finite, since the system is not conservative: there is a
special vertex that serves as a sink, absorbing any sand that reaches it.  It is
assumed that every vertex is connected to the sink by a path of edges, so as a
consequence, every pile of sand placed on the graph stabilizes after a finite
number of vertex firings.  It turns out that this stable state only depends on
the initial configuration of sand, not on the order of the firings of unstable
vertices, which accounts for the use of the word ``abelian.''

Now imagine starting with no sand on $\Gamma$ then repeatedly choosing a vertex
at random, adding a grain of sand, and allowing the pile of sand to stabilize.
In the resulting sequence of configurations of sand, certain configurations will
appear infinitely often.  These are the so-called ``recurrent'' configurations.
A basic theorem in the theory of sandpiles is that the collection of recurrent
configurations forms an additive group, where addition is defined as vertex-wise
addition of grains of sand, followed by stabilization.  This group is called the
{\em sandpile group} or {\em critical group} of $\Gamma$.  Equivalent versions
of the sandpile group have arisen independently.  For a history and as a general
reference, see~\cite{Holroyd}.

In their seminal 1987 paper, Bak, Tang, and Wiesenfeld (BTW),~\cite{BTW}, studied
sandpile dynamics in the case of what we call the {\em sandpile grid graph}.  To
construct the $m\times n$ sandpile grid graph, start with the ordinary grid
graph with vertices $[m]\times[n]$ and edges $\{(i,j),(i',j')\}$ such that
$|i-i'|+|j-j'|=1$.  Then add a new vertex to serve as a sink, and add edges from
the boundary vertices to the sink so that each vertex on the grid has
degree~$4$.  Thus, corner vertices have two edges to the sink (assuming $m$ and
$n$ are greater than~$1$), as on the left in Figure~\ref{fig:sandpile grid
graph}.  Dropping one grain of sand at a time onto a sandpile grid graph and
letting the system stabilize, BTW experimentally finds that eventually the
system evolves into a barely stable ``self-organized critical'' state.  This
critical state is characterized by the property that the sizes of avalanches
caused by dropping a single grain---measured either temporally (by the number of
ensuing vertex firings) or spatially (by the number of different vertices that
fire)---obey a power law. The power-laws observed by BTW in the case of some
sandpile grid graphs have not yet been proven.

The ASM, due to  Dhar~\cite{Dhar1}, is a generalization of the BTW model to a
wider class of graphs.  It was Dhar who made the key observation of its abelian
property and who coined the term ``sandpile group'' for the collection of
recurrent configurations of sand.   In terms of the ASM, the evolution to a
critical state observed by BTW comes from the fact that by repeatedly adding a
grain of sand to a graph and stabilizing, one eventually reaches a configuration
that is recurrent.  Past this point, each configuration reached by adding sand
and stabilizing is again recurrent.

Other work on ASM and statistical physics includes~\cite{Fey}.  In addition, the
ASM has been shown to have connections with a wide range of mathematics,
including algebraic geometry and commutative algebra
(\cite{Baker},~\cite{Dochtermann},~\cite{Payne},~\cite{Madhu},~\cite{Madhu2},~\cite{Wilmes}),
pattern formation (\cite{Ostojic},\cite{Paoletti},\cite{Pedgen1},
\cite{Pedgen2},\cite{Sadhu}), potential theory
(\cite{Baker2},\cite{Biggs},\cite{Levine}), combinatorics
(~\cite{Hopkins2},~\cite{Hopkins1},\cite{Merino},~\cite{Postnikov}), and number
theory (\cite{Musiker}).  The citations here are by no means exhaustive.  One
might argue that the underlying reason for these connections is that the firing
rules for the ASM encode the discrete Laplacian matrix of the graph (as
explained in Section~\ref{section:Sandpiles}).  Thus, the ASM is a means of
realizing the dynamics implicit in the discrete Laplacian.
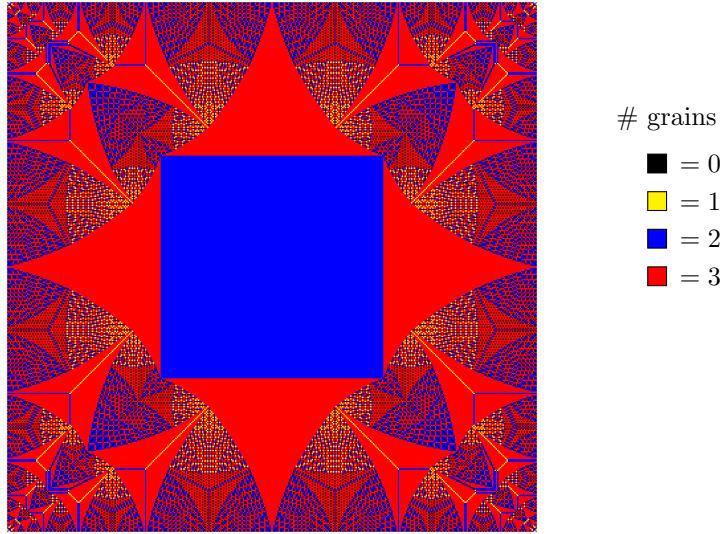
\begin{figure}[ht]
\begin{tikzpicture}[scale=0.5]
\def\i{10};
\def\j{3};
\node at (\i+0.5,\j+0.95){\# grains};
\node at (0,0){\includegraphics[height=3.0in]{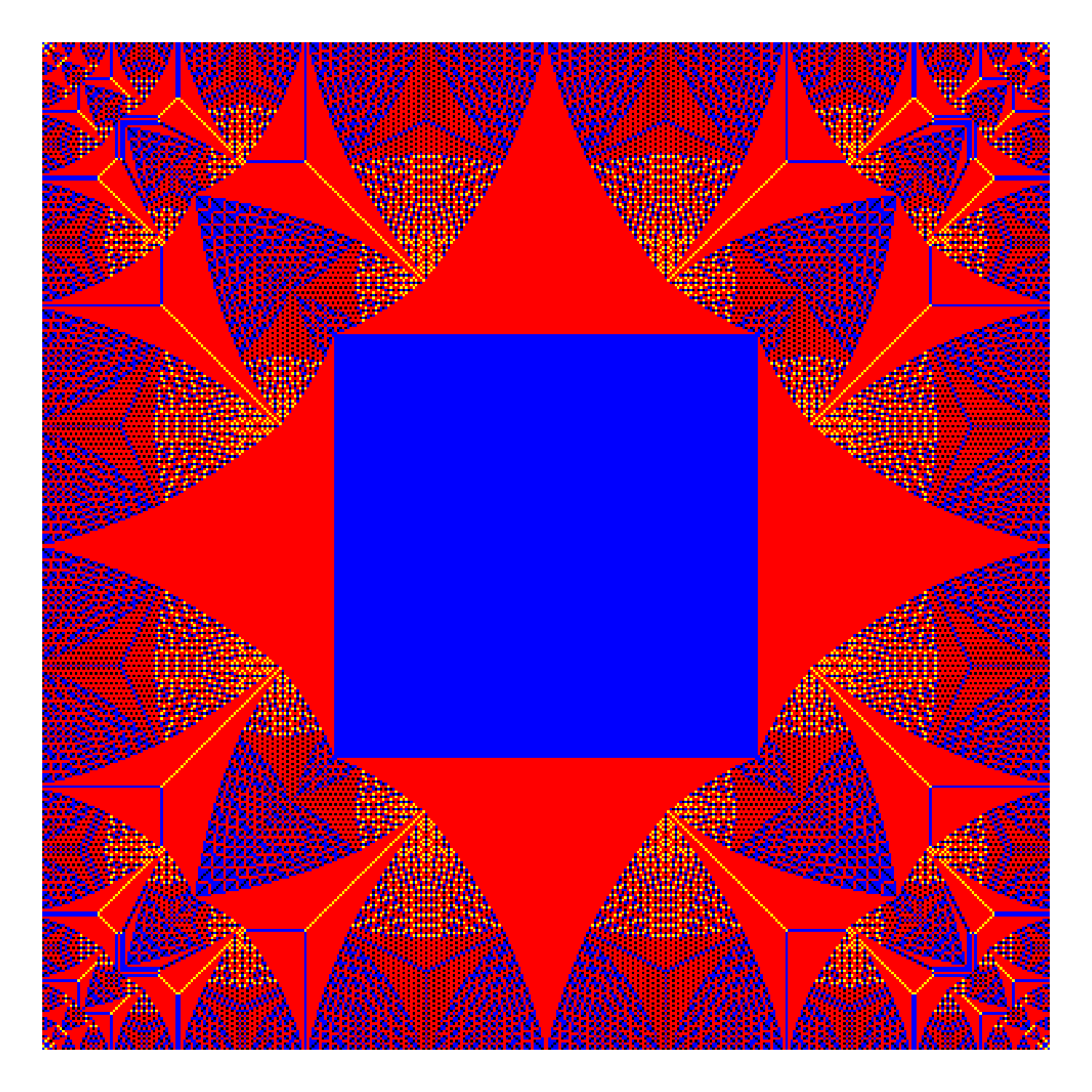}};
\draw[fill=black] (\i,\j) rectangle (\i+0.5,\j-0.5);
\node at (\i+1.4,\j-0.25){$=0$};
\draw[fill=yellow] (\i,\j-1) rectangle (\i+0.5,\j-1.5);
\node at (\i+1.4,\j-1.25){$=1$};
\draw[fill=blue] (\i,\j-2) rectangle (\i+0.5,\j-2.5);
\node at (\i+1.4,\j-2.25){$=2$};
\draw[fill=red] (\i,\j-3) rectangle (\i+0.5,\j-3.5);
\node at (\i+1.4,\j-3.25){$=3$};
\end{tikzpicture}
\caption{Identity element for the sandpile group of the $400\times400$
sandpile grid graph.}\label{fig:400x400identity}
\end{figure}

The initial motivation for our work was a question posed to the second and third
authors by Irena Swanson.  She was looking at an online computer
program~\cite{Maslov} for visualizing the ASM on a sandpile grid graph.  By
pushing a button, the program adds one grain of sand to each of the nonsink
vertices then stabilizes the resulting configuration.  Swanson asked, ``Starting
with no sand, how many times would I need to push this button to get the
identity of the sandpile group?'' A technicality arises here: the configuration
consisting of one grain of sand on each vertex is not recurrent, hence, not in
the group.  However, the all-$2$s configuration, having two grains at each
vertex, is recurrent.  So for the sake of this introduction, we reword the
question as: ``What is the order of the all-$2$s configuration?''

Looking at data (cf.~Section~\ref{section:order of all-2s},
Table~\ref{table:all-2s}), one is naturally led to the special case of the
all-$2$s configuration on the $2n\times2n$ sandpile grid graph, which we denote
by $\vec{2}_{2n\times 2n}$.  The orders for $\vec{2}_{2n\times2n}$ for
$n=1,\dots,5$ are
\[
1,3,29,901,89893.
\]
Plugging these numbers into the Online Encyclopedia of Integer Sequences
yields a single match, sequence A065072 (\cite{OEIS}): the sequence of odd integers
$(a_n)_{n\geq1}$ such that $2^na_n^2$ is the number of domino tilings of the
$2n\times 2n$ checkerboard.\footnote{By a {\em checkerboard} we mean a rectangular array of squares.  A {\em domino} is a $1\times2$ or $2\times1$ array of squares.  A {\em domino tiling} of a
checkerboard consists of covering all of the squares of the checkerboard---each
domino covers two---with dominos.}
 (Some background on this sequence is included in
Section~\ref{section:order of all-2s}.)  So we conjectured that the order of
$\vec{2}_{2n\times2n}$ is equal to $a_n$, and trying to prove this is what first
led to the ideas presented here. Difficulty in finishing our
proof of the conjecture led to further computation, at which time we
(embarrassingly) found that the order of $\vec{2}_{2n\times2n}$ for $n=6$ is,
actually, $5758715=a_6/5$.  Thus, the conjecture is false, and there are
apparently at least two natural sequences that start $1,3,29,901, 89893$!
Theorem~\ref{thm4} shows that the cyclic group generated
by $\vec{2}_{2n\times2n}$ is isomorphic to a subgroup of a sandpile group whose
order is $a_n$, and therefore the order of $\vec{2}_{2n\times2n}$ divides $a_n$.
We do not know when equality holds, and we have not yet answered Irena Swanson's question.

On the other hand, further experimentation using the mathematical software Sage
led us to a more fundamental connection between the sandpile group and domino
tilings of the grid graph.  The connection is due to a property that is a
notable feature of the elements of the subgroup generated by the all-$2$s
configuration---{\em symmetry} with respect to the central horizontal and
vertical axes.  The recurrent identity element for the sandpile grid graph, as
exhibited in Figure~\ref{fig:400x400identity}, also has this symmetry.\footnote{For
square grids, the identity is symmetric with respect to the dihedral group of
order $8$, but this phenomenon is of course not present in the rectangular grids
that we also consider.}  If $\Gamma$ is any graph equipped with an
action of a finite group~$G$, it is natural to consider the collection of
$G$-invariant configurations. Proposition~\ref{prop:symmetric configs}
establishes that the symmetric recurrent configurations form a subgroup of the
sandpile group for $\Gamma$.  The central purpose of this paper is to explain
how symmetry links the sandpile group of the grid graph to domino tilings.

We now describe our main results.  We study the recurrent configurations on the sandpile grid graph having $\Z/2 \times \Z/2$ symmetry with respect to the central horizontal and vertical axes. The cases of even$\times$even-, even$\times$odd-, and
odd$\times$odd-dimensional grids each have their own particularities, and so we divide their analysis into separate cases, resulting in Theorems~\ref{thm1},~\ref{thm2},
and~\ref{thm3}, respectively.  In each case, we compute the number of symmetric
recurrents as (i) the number of domino tilings of corresponding
(weighted) rectangular checkerboards; (ii) a product of special values of Chebyshev
polynomials; and (iii) a double-product whose factors are sums of squares of
values of trigonometric functions.

For instance, of the $557,568,000$ elements of the sandpile group of the $4
\times 4$ grid graph, only the $36$ configurations displayed in
Figure~\ref{fig:checker4x4} are up-down and left-right symmetric.  In accordance
with Theorem~\ref{thm1},
\begin{align}\label{intro:double product}
36&=U_4(i\cos(\pi/5))\,U_4(i\cos(2\pi/5))\notag\\[5pt]
&=\prod_{h=1}^{2}\prod_{k=1}^2\left(4\,\cos^2(h\pi/5)+4\,\cos^2(k\pi/5)\right),
\end{align}
where $U_4(x)=16x^4-12x^2+1$ is the fourth Chebyshev polynomial of the second
kind.

The double-product in equation~\eqref{intro:double product} is an instance of
the famous formula due to Kasteleyn~\cite{Kasteleyn} and to Temperley and
Fisher~\cite{Temperley} for the number of domino tilings of a $2m\times 2n$
checkerboard:
\[
\prod_{h=1}^m\prod_{k=1}^n\left(4\,\cos^2\frac{h\pi}{2m+1}+
4\,\cos^2\frac{k\pi}{2n+1}\right),
\]
for which Theorem~\ref{thm1} provides a new proof.

In the case of the even$\times$odd grid, there is an extra ``twist'': the
double-product in Theorem~\ref{thm2} for the even$\times$odd grid is (a slight
re-writing of) the formula of Lu and Wu~\cite{LW} for the number of domino
tilings of a checkerboard on a M\"obius strip.

To sketch the main idea behind the proofs of these theorems, suppose a group~$G$
acts on a graph $\Gamma$ with fixed sink vertex
(cf.~Section~\ref{subsection:Symmetric configurations}).  To study symmetric
configurations with respect to the action of $G$, one considers a new firing
rule in which a vertex only fires simultaneously with all other vertices in its
orbit under~$G$.  This new firing rule can be encoded in an $m\times m$ matrix
$D$ where $m$ is the number of orbits of nonsink vertices of $G$.  We show in
Corollary~\ref{cor:nsr} that $\det(D)$ is the number of symmetric recurrents on
$G$.  Suppose, as is the case for for sandpile grid graphs, that either $D$ or
its transpose happens to be the (reduced) Laplacian of an associated graph
$\Gamma'$. The nonsink vertices correspond to the orbits of vertices of the
original graph.  The well-known matrix-tree theorem says that the determinant of
$D$ is the number of spanning trees of $\Gamma'$.  Then the generalized
Temperley bijection~\cite{KPW} says these spanning trees correspond with perfect
matchings of a third graph $\Gamma''$.  In this way, the symmetric recurrents on
$\Gamma$ can be put into correspondence with the perfect matchings of
$\Gamma''$.  In the case where $\Gamma$ is a sandpile grid graph, $\Gamma''$ is
a weighted grid graph, and perfect matchings of it correspond to weighted
tilings of a checkerboard.  Also, in this case, the matrix $D$ has a nice block
triangular form (cf.~Lemma~\ref{lemma:tridiagonal}), which leads to a recursive
formula for its determinant and a connection with Chebyshev polynomials.
\newpage

\centerline{\sc outline}
\medskip

\hangindent1em
\hangafter=0
\noindent\ref{section:Introduction} {\bf Introduction.}\vskip4pt

\hangindent1em
\hangafter=0
\noindent\ref{section:Sandpiles} {\bf Sandpiles.}\vskip2pt

\hangindent2em
\hangafter=0
\noindent\ref{subsection:Basics} {\bf Basics.}  A summary of the basic theory of
sandpiles needed for this paper.\vskip2pt

\hangindent2em
\hangafter=0
\noindent\ref{subsection:Symmetric configurations} {\bf Symmetric
configurations.}  Group actions on sandpile graphs.
Proposition~\ref{prop:symmetric configs} shows that the collection of symmetric
recurrents forms a subgroup of the sandpile group.  We introduce the {\em
symmetrized reduced Laplacian} operator,~$\tD^G$, and use it to determine the structure of
this subgroup in Proposition~\ref{prop:symmetric subgroup iso}.  An important
consequence is Corollary~\ref{cor:nsr}, which shows that the number of
symmetric recurrents equals $\det\tD^G$.\vskip4pt

\hangindent1em
\hangafter=0
\noindent\ref{section:Matchings and trees} {\bf Matchings and trees.}  A
description of the generalized Temperley bijection~\cite{KPW} between
weighted spanning trees of a planar graph and weighted perfect matchings of a
related graph.\vskip4pt

\hangindent1em
\hangafter=0
\noindent\ref{section:symmetric recurrents} {\bf Symmetric recurrents on the sandpile grid
graph.} We count symmetric recurrents on sandpile grid graphs using weighted
tilings of checkerboards, Chebyshev polynomials, and Kasteleyn-type formulae.
The problem is split into three cases.\vskip2pt

\hangindent2em \hangafter=0 \noindent\ref{subsection:tridiagonal} {\bf Some
tridiagonal matrices.} A summary of some properties of Chebyshev polynomials and
a proof Lemma~\ref{lemma:tridiagonal}, which calculates the determinant of a
certain form of tridiagonal block matrix.  The symmetrized reduced Laplacian
matrices for the three classes of sandpile grid graphs, below, have this form.
Their determinants count symmetric recurrents.\vskip2pt

\hangindent2em
\hangafter=0
\noindent\ref{subsection:symmetric recurrents on evenxeven grid}  {\bf Symmetric
recurrents on a $2m\times 2n$ sandpile grid graph.}  See Theorem~\ref{thm1}.\vskip2pt

\hangindent2em
\hangafter=0
\noindent\ref{subsection:symmetric recurrents on evenxodd grid}  {\bf Symmetric
recurrents on a $2m\times(2n-1)$ sandpile grid graph.} See
Theorem~\ref{thm2}.\vskip2pt

\hangindent2em
\hangafter=0
\noindent\ref{subsection:symmetric recurrents on oddxodd grid}  {\bf Symmetric
recurrents on a $(2m-1)\times(2n-1)$ sandpile grid graph.} See
Theorem~\ref{thm3}.\vskip4pt

\hangindent1em \hangafter=0 \noindent\ref{section:order of all-2s} {\bf The
order of the all-twos configuration.} Corollary~\ref{cor:all-2s order}: the
order of the all-$2$s configuration on the $2n\times2n$ sandpile grid divides
the odd number $a_n$ such that $2^na_n^2$ is the number of domino tilings of the
$2n\times2n$ checkerboard.\vskip4pt

\hangindent1em
\hangafter=0
\noindent\ref{section:conclusion} {\bf Conclusion.}  A list of open problems.
\medskip

\noindent{\bf Acknowledgments.}  We thank Irena Swanson for providing initial
motivation.  We thank the organizers of the Special Session on Laplacian Growth
at the Joint Mathematics Meeting, New Orleans, LA, 2011 at which some of this
work was presented, and we thank Lionel Levine, in particular, for encouragement
and helpful remarks.  Finally, we would like to acknowledge the mathematical
software Sage~\cite{sage} and the Online Encyclopedia of Integer
Sequences~\cite{OEIS} which were both essential for our investigations.
 
\section{Sandpiles}\label{section:Sandpiles}
\subsection{Basics}\label{subsection:Basics} In this section, we recall the basic
theory of sandpile groups.  The reader is referred to \cite{Holroyd} for
details.  Let $\Gamma=(V,E,\wt,s)$ be a directed graph with vertices $V$, edges
$E$, edge-weight function $\wt\colon V\times V\to \N:=\{0,1,2,\dots\}$, and
special vertex $s\in V$.  For each pair $v,w\in V$, we think of $\wt(v,w)$ as
the number of edges running from $v$ to $w$.  In particular, $\wt(v,w)>0$ if and
only if $(v,w)\in E$.  The vertex $s$ is called the {\em sink} of $\Gamma$, and
it is assumed that each vertex of $\Gamma$ has a directed path to $s$. Let $\tV
:= V\setminus\{s\}$ be the set of non-sink vertices.  A {\em (sandpile)
configuration} on $\Gamma$ is an element of $\N\tV$, the free monoid on $\tV$.
If $c=\sum_{v\in\tV}c_v\,v$ is a configuration, we think of each component,
$c_v$, as a number of grains of sand stacked on vertex $v$.  The vertex
$v\in\tV$ is {\em unstable} in $c$ if $c_v\geq \outdeg(v)$ where
$\outdeg(v):=\sum_{w\in V}\wt(v,w)$, is the {\em out-degree} of $v$, i.e., the
number of directed edges emanating from~$v$.  If $v$ is unstable in $c$, we may
{\em fire} ({\em topple}) $c$ at $v$ to get a new configuration $c'$ defined for
each $w\in\tV$ by
\[
c'_w=
\begin{cases}
  c_v-\outdeg(v)+\wt(v,v)&\mbox{if $w=v$,}\\
  c_w+\wt(v,w)&\mbox{if $w\neq v$}.
\end{cases}
\]
In other words,
\[
c' = c - \outdeg(v)v+\textstyle\sum_{w\in\tV}\wt(v,w)\,w.
\]
If the configuration $\tilde{c}$ is obtained from $c$ by a sequence of firings
of unstable vertices, we write
\[
c\to\tilde{c}.
\]
Since each vertex has a path to the sink, $s$, it turns out that by repeatedly
firing unstable vertices each configuration relaxes to a stable configuration.
Moreover, this stable configuration is independent of the ordering of firing of
unstable vertices. Thus, we may talk about {\em the} stabilization of a
configuration $c$, which we denote by~$c^{\circ}$.  Define the binary operation
of {\em stable addition} on the set of all configurations as component-wise
addition followed by stabilization.  In other words, the stable addition of
configurations $a$ and $b$ is given by
\[
(a+b)^{\circ}.
\]
Let $\mathcal{M}$ denote the collection of stable
configurations on $\Gamma$. Then stable addition restricted to $\mathcal{M}$
makes $\mathcal{M}$ into a commutative monoid. 

A configuration $c$ on $\Gamma$ is {\em recurrent} if: (1) it is stable, and (2)
given any configuration $a$, there is a configuration $b$ such that
$(a+b)^{\circ}=c$.  The {\em maximal stable configuration}, $\cmax$, is defined
by
\[
\cmax:=\sum_{v\in\tV}(\outdeg(v)-1)\,v.
\]
It turns out that the collection of recurrent configurations forms
a principal semi-ideal of $\mathcal{M}$ generated by $c_{\mathrm{max}}$.  This
means that the recurrent configurations are exactly those obtained by adding
sand to the maximal stable configuration and stabilizing.  Further, the
collection of recurrent configurations forms a group, $\sand(\Gamma)$,
called the {\em sandpile group} for $\Gamma$.  Note that the identity for
$\mathcal{S}(\Gamma)$ is not usually the zero-configuration, $\vec{0}\in\N\tV$.

For an undirected graph, i.e., a graph for which $\wt(u,v)=\wt(v,u)$ for each
pair of vertices $u$ and $v$, one may use the {\em burning algorithm}, due to
Dhar~\cite{Dhar2}, to determine whether a configuration is recurrent (for a
generalization to directed graphs, see~\cite{Speer}):
\begin{thm}[{\cite{Dhar2},\cite[Lemma~4.1]{Holroyd}}]\label{thm:Dhar}  Let $c$ be
  a stable configuration on an undirected graph $\Gamma$.  Define the {\em burning configuration}
  on~$\Gamma$ to be the configuration obtained by {\em firing the sink vertex}:
  \[
  b:=\sum_{v\in\tV}\wt(s,v)\,v.
  \]
  Then in the stabilization of $b+c$, each vertex fires at most once, and
  the following are equivalent:
  \begin{enumerate}
    \item $c$ is recurrent;
    \item $(b+c)^{\circ}=c$;
    \item in the stabilization of $b+c$, each non-sink vertex fires.
  \end{enumerate}
\end{thm}

Define the {\em proper Laplacian}, $L\colon \Z^V\to\Z^V$, of $\Gamma$ by
\[
L(f)(v):=\sum_{w\in V}\wt(v,w)(f(v)-f(w))
\]
for each function $f\in\Z^V$.  Taking the $\Z$-dual (applying the functor
$\mathrm{Hom}(\ \cdot\ ,\Z))$ gives the mapping of free abelian groups
\[
\Delta\colon\Z V\to\Z V
\]
defined on vertices $v\in V$ by
\[
\Delta(v)=\outdeg(v)\,v-\sum_{w\in V}\wt(v,w)\,w.
\]
We call $\Delta$ the {\em Laplacian} of $\Gamma$.  Restricting $\Delta$ to
$\Z\tV$ and setting the component of~$s$ equal to $0$ gives the {\em reduced
Laplacian}, $\tD\colon\Z\tV\to\Z\tV$.  If $v$ is an unstable vertex in a
configuration $c$, firing $v$ gives the new configuration
\[
c-\tD v.
\]

There is a well-known isomorphism
\begin{align}\label{basic iso}
  \mathcal{S}(\Gamma)&\to\Z\tV/\mathrm{image}(\tD)\\
  c&\mapsto c.\nonumber
\end{align} 
While there may be many stable configurations in each equivalence class of
$\Z\tV$ modulo $\mathrm{image}(\tD)$, there is only one that is recurrent.  For
instance, the recurrent element in the equivalence class of $\vec{0}$ is the
identity of $\mathcal{S}(\Gamma)$.

A {\em spanning tree of $\Gamma$ rooted at $s$} is a directed subgraph
containing all the vertices, having no directed cycles, and for which $s$ has
no out-going edges while every other vertex has exactly one out-going edge.  The
weights of the edges of a spanning tree are the same as they
are for $\Gamma$, and the {\em weight} of a spanning tree is the product of the
weights of its edges.  The matrix-tree theorem says the sum of the weights of
the set of all spanning trees of $\Gamma$ rooted at $s$ is equal to $\det\tD$, the
determinant of the reduced Laplacian.  It then follows from (\ref{basic iso})
that the number of elements of the sandpile group is also the sum of the weights
of the spanning trees rooted at $s$.

\subsection{Symmetric configurations}\label{subsection:Symmetric configurations} 

Preliminary versions of the results in this section appear in \cite{Durgin}.
Let $G$ be a finite group.  An {\em action} of $G$ on $\Gamma$ is an action of
$G$ on $V$ fixing $s$, sending edges to edges, and preserving edge-weights.  In
detail, it is a mapping
\begin{eqnarray*}
  G\times V&\to&V\\
  (g,v)&\mapsto&gv
\end{eqnarray*}
satisfying
\begin{enumerate}
  \item if $e$ is the identity of $G$, then $ev=v$ for all $v\in V$;
  \item $g(hv)=(gh)v$ for all $g,h\in G$ and $v\in V$;
  \item $gs=s$ for all $g\in G$;
  \item if $(v,w)\in E$, then $(gv,gw)\in E$ and both edges have the same
    weight.
\end{enumerate}
Note that these conditions imply that $\outdeg(v)=\outdeg(gv)$ for all
$v\in V$ and $g\in G$.
For the rest of this section, let $G$ be a group acting on $\Gamma$.

By linearity, the action of $G$ extends to an action on $\N V$ and $\Z V$. 
Since~$G$ fixes the sink, $G$ acts on configurations and each element of
$G$ induces an automorphism of $\mathcal{S}(\Gamma)$ (cf.~\ref{cor:preserved}).
We say a configuration $c$ is {\em symmetric} (with respect to the action
by $G$) if $gc=c$ for all $g\in G$.

\begin{prop}
  The action of $G$ commutes with stabilization.  That is, if $c$ is any
  configuration on $\Gamma$ and $g\in G$, then $g(c^{\circ})=(gc)^{\circ}$.
\end{prop}
\begin{proof}  Suppose that $c$ is stabilized by firing the sequence of vertices
  $v_1,\dots,v_t$.  Then
  \[
  c^{\circ}=c-\sum_{i=1}^t\tD v_i.
  \]
  At the $k$-th step in the stabilization process, $c$ has relaxed to the
  configuration $c':=c-\sum_{i=1}^k\tD v_i$.  A vertex $v$ is unstable in
  $c'$ if and only if $gv$ is unstable in
  $gc'=gc-\sum_{i=1}^k\tD (gv_i)$.  Thus, we can fire the sequence of
  vertices $gv_1,\dots,gv_t$ in~$gc$, resulting in the stable configuration
  \[
  (gc)^{\circ}=gc-\sum_{i=1}^t\tD (g v_i).
  \]
\end{proof}
\begin{cor}\label{cor:preserved}
  The action of $G$ preserves recurrent configurations, i.e., if $c$
is a recurrent configuration and $g\in G$, then $gc$
  is recurrent. 
\end{cor}
\begin{proof}
  If $c$ is recurrent, we can find a configuration $b$ such that
  $c=(b+\cmax)^{\circ}$.  Then,
  \[
  gc=g(b+\cmax)^{\circ}=(gb+g\cmax)^{\circ}=(gb+\cmax)^{\circ}.
  \]
  Hence, $gc$ is recurrent.
\end{proof}
\begin{cor}\label{cor:symm-stab}
  If $c$ is a symmetric configuration, then so is its stabilization.
\end{cor}
\begin{proof}
  For all $g\in G$, if $gc=c$, then
  $g(c^{\circ})=(gc)^{\circ}=c^{\circ}$.
\end{proof}
\begin{remark}
  In fact, if $c$ is a symmetric configuration, one may find a sequence of
  symmetric configurations, $c_1, \dots,c_t$ with $c_t=c^{\circ}$ such that
  $c\to c_1\to\cdots\to c_t$.  This follows since in a symmetric configuration
  a vertex $v$ is unstable if and only if $gv$ is unstable for all $g\in G$.  To
  construct $c_{i+1}$ from $c_i$, simultaneously fire all unstable vertices of
  $c_i$ (an alternative is to pick any vertex $v$, unstable in $c_{i}$, and
  simultaneously fire the vertices in $\{gv:g\in G\}$).
\end{remark}
\begin{prop}\label{prop:symmetric configs}
  The collection of symmetric recurrent configurations forms a subgroup of the
  sandpile group $\mathcal{S}(\Gamma)$.  
\end{prop}
\begin{proof}
  Since the group action respects addition in $\N\tV$ and
  stabilization, the sum of two symmetric recurrent configurations is again
  symmetric and recurrent.  There is at least one symmetric recurrent
  configuration, namely, $\cmax$.  Since the sandpile group is
  finite, it follows that these configurations form a subgroup.  
\end{proof}
\medskip

\begin{notation}
The subgroup of symmetric recurrent configurations on  $\Gamma$ with respect to
the action of the group $G$ is denoted $\mathcal{S}(\Gamma)^G$.
\end{notation}
\begin{prop}
  If $c$ is symmetric and recurrent then $c=(a+\cmax)^{\circ}$ for some
  symmetric configuration~$a$.
\end{prop}
\begin{proof} By \cite{Speer} there exists an element $b$ in the image of
  $\tD$ such that: (1) $b_v\geq0$ for all $v\in\tV$, and (2) for each
  vertex $w\in\tV$, there is a directed path to $w$ from some $v\in\tV$ such that
  $b_v>0$, i.e., from some $v$ in the {\em support} of $b$.  (If $\Gamma$ is
  undirected, one may find such a $b$ by applying
  $\tD$ to the vector whose components are all~$1$s).  Define
\[
b^G=\sum_{g\in G}gb.
\]
Then $b^G$ is symmetric and equal to zero modulo the image of $\tD $.  Take a
large positive integer $N$ and consider $Nb^G$, the vertex-wise addition of $b^G$
with itself~$N$ times without stabilizing.  Every vertex of $\Gamma$ is
connected by a path from a vertex in the support of~$b$, and hence, the same is
true of $Nb^G$.  Thus, by choosing~$N$ large enough and by firing symmetric
vertices of $Nb^G$, we obtain a symmetric configuration $b'$ such that $b'_v\geq
c_{\mathrm{max},v}$ for all $v$ and such that $b'$ is zero modulo the image of
$\tD$.  Define $a=b'-\cmax+c$, by construction a symmetric configuration.
The unique recurrent element in the equivalence
class of $b'+c$ modulo the image of $\tD$ is $c$.  Therefore,
\[
(a+\cmax)^{\circ}=(b'+c)^{\circ}=c.
\]
\end{proof}

The {\em orbit} of $v\in V$ under~$G$ is the set
\[
Gv=\{gv:g\in G\}.
\]
Let $\mathcal{O}=\mathcal{O}(\Gamma,G)=\{Gv:v\in \tV\}$ denote the set of orbits
of the non-sink vertices.  The {\em symmetrized reduced Laplacian} is the $\Z$-linear mapping 
\begin{equation}
\tD^G\colon\Z\mathcal{O}\to\Z\mathcal{O}
\end{equation}\label{eqn:srl}
such that for all $v,w\in\tV$, the $Gw$-th component of $\tD^G(Gv)$
 is
\[
\left(\textstyle\sum_{u\in Gv}\tD(u)\right)_w.
\]
\begin{remark}\label{orbit correspondence} If $c\in\Z\tV$ is symmetric, then
  define $[c]\in\Z\mathcal{O}$ by $[c]_{Gv}:=c_v$ for all $v\in\tV$, thus
  obtaining a bijection between symmetric elements of $\Z\tV$ and
  $\Z\mathcal{O}$.  
The mapping $\tD^G$ is defined so that if $c$ is a symmetric
  configuration and $v\in \tV$, then $[c]-\tD^G(Gv)$ is the element of $\Z\mathcal{O}$
  corresponding to 
  \[
  c-\tD(\textstyle\sum_{w\in Gv}w),
  \]
  the symmetric configuration obtained from $c$ by firing all vertices in the
  orbit of $v$.
\end{remark}

For the following let $r\colon\Z\tV/\mathrm{image}(\tD)\to\mathcal{S}(\Gamma)$
  denote the inverse of the isomorphism in (\ref{basic iso}).
\begin{prop}\label{prop:symmetric subgroup iso} There is an isomorphism of groups,
  \[
  \phi\colon\Z\mathcal{O}/\mathrm{image}(\tD^G)\to\mathcal{S}(\Gamma)^G,
  \]
  determined by $Gv\mapsto r(\textstyle\sum_{w\in Gv}w)$ for $v\in\tV$.
\end{prop}
\begin{proof}
  The homomorphism $\lambda\colon\Z\mathcal{O}\to\Z\tV$ determined by
  \[
  \lambda(Gv):=\sum_{w\in Gv}w
  \]
  for $v\in\tV$ induces the (well-defined) mapping
\[
\Lambda\colon\Z\mathcal{O}/\mathrm{image}(\tD^G)\to\Z\tV/\mathrm{image}(\tD).
\] 
To see that the image of $r\circ\Lambda$ is symmetric, consider
the symmetric configuration $|\mathcal{S}(\Gamma)|\cdot\cmax\in\Z\tV$, a
configuration in the image of $\tD$.  For
each $v\in\tV$,
\[
\phi(Gv)=r(\Lambda(Gv))=\left(\,|\mathcal{S}(\Gamma)|\cdot\cmax+{\lambda}(Gv)\right)^{\circ},
\]
which is symmetric by Corollary~\ref{cor:symm-stab}.  

The mapping $c\mapsto[c]$, introduced in Remark \ref{orbit correspondence}, is a
left inverse to $\lambda$.  Thus, if $c\in\mathcal{S}(\Gamma)^G$,  then
$\phi([c])=c$, and hence $\phi$ is surjective.  To show that $\phi$ is
injective, it suffices to show that $\Lambda$ is injective.  So suppose that
$a=\lambda(o)$ for some $o\in\Z\mathcal{O}$ and that
$a=\tD(b)$ for some $b\in\Z\tV$.  Fix $g\in G$, and consider the isomorphism
$g\colon \Z\tV\to\Z\tV$ determined by the action of $g$ on vertices.  A
straightforward calculation shows that $\tD=g\tD g^{-1}$.  It follows that
\[
\tD(b)= a = ga = g\tD b = (g\tD g^{-1})(gb) = \tD(gb).
\]
Since $\tD$ is invertible, it follows that $b=gb$ for all $g\in G$, i.e., $b$ is
symmetric.  Hence, $o=[a]=\tD^G([b])$, as required.
\end{proof}
\begin{cor}\label{cor:nsr}  The number of symmetric recurrent configurations is
\[
|\mathcal{S}(\Gamma)^G|=\det\tD^G.
\]
\end{cor}
\begin{remark}
We have not assumed that the action of $G$ on $\Gamma$ is faithful.  If $K$ is
the kernel of the action of $G$, then
$\mathcal{O}(\Gamma,G)=\mathcal{O}(\Gamma,G/K)$
and $\mathcal{S}^G=\mathcal{S}^{G/K}$.  We also have $\tD^G=\tD^{G/K}$.
\end{remark}

\begin{example}\label{symmetry example}
  Consider the graph $\Gamma$ of Figure~\ref{fig:symmetry example} with sink $s$
  and with each edge having weight~$1$.

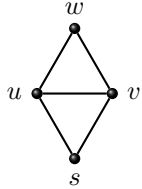
\begin{figure}[ht] 
\begin{tikzpicture}[scale=1.0]
\SetVertexMath
\GraphInit[vstyle=Art]
\SetUpVertex[MinSize=3pt]
\SetVertexLabel
\tikzset{VertexStyle/.style = {%
shape = circle,
shading = ball,
ball color = black,
inner sep = 1.5pt
}}
\SetUpEdge[color=black]

\Vertex[LabelOut=true,Lpos=180,x=0,y=0]{u}
\Vertex[LabelOut=true,Lpos=0,x=1,y=0]{v}
\Vertex[LabelOut=true,Lpos=90,x=0.5,y=0.866]{w}
\Vertex[LabelOut=true,Lpos=270,x=0.5,y=-0.866]{s}
\Edges(u,s,v,w,u)
\Edge[](u)(v)
\end{tikzpicture}
\caption{The graph $\Gamma$ for Example~\ref{symmetry
example}.}\label{fig:symmetry example}
\end{figure}

Let $G=\{e,g\}$ be the group of order $2$ with identity $e$. Consider the action
of~$G$ on $\Gamma$ for which $g$ swaps vertices $u$ and $v$ and fixes vertices
$w$ and $s$.  Ordering the vertices of $\Gamma$ as $u,v,w$ and ordering the
orbits, $\mathcal{O}$, as $Gu$, $Gw$, the reduced Laplacian and the symmetrized
reduced Laplacian for $\Gamma$ become
\begin{center}
  \begin{tikzpicture}
\draw (0,0) node{$
\tD=
\left[
\begin{array}{rrr}
3&-1&-1\\
-1&3&-1\\
-1&-1&2
\end{array}
\right]$,};
\draw (-0.4,-0.9) node{$u$};
\draw (0.4,-0.9) node{$v$};
\draw (1.2,-0.9) node{$w$};
\draw (5,0) node{$
\tD^G=
\left[
\begin{array}{rr}
2&-1\\
-2&2
\end{array}
\right]$,};
\draw (5.1,-0.8) node{$Gu$};
\draw (6.0,-0.8) node{$Gw$};
  \end{tikzpicture}
\end{center}
where we have labeled the columns by their corresponding vertices or orbits for
convenience.  To illustrate how one would compute the columns of the symmetrized
reduced Laplacian in general, consider the column of $\tD^G$ corresponding to
$Gu=\{u,v\}$.  It was computed by first adding the $u$- and $v$-columns of $\tD$
to get the $3$-vector $\ell=(2,2,-2)$,  then taking the $u$ and $w$ components of $\ell$
since $u$ and $w$ were chosen as orbit representatives. 

There are $8=\det\tD$ recurrent elements $(c_u,c_v,c_w)$ of $\Gamma$:
\[
(0,2,1),(1,2,0),(1,2,1),(2,0,1),(2,1,0),(2,1,1),(2,2,0),(2,2,1),
\]
and $(2,2,0)$ is the identity of $\sand(\Gamma)$. In accordance with
Corollary~\ref{cor:nsr}, there are $2=\det\tD^G$ symmetric recurrent elements:
$(2,2,0)$ and $(2,2,1)$. 
\end{example}
\section{Matchings and trees}\label{section:Matchings and trees}
In this section, assume that $\Gamma=(V,E,\wt,s)$ is embedded in the plane, and
fix a face $f_s$ containing the sink vertex,
$s$.  In \S\ref{section:symmetric recurrents} and~\S\ref{section:order
of all-2s}, we always take $f_s$ to be the unbounded face. We recall the generalized Temperley bijection, due to \cite{KPW}, between
directed spanning trees of~$\Gamma$ rooted at $s$ and perfect matchings of a
related weighted undirected graph,~$\tG$.  (The graph $\tG$ would be denoted
$\mathcal{H}(s,f_s)$ in \cite{KPW}.)

It is sometimes convenient to allow an edge $e=(u,v)$ to be represented in the embedding
by distinct weighted edges $e_1,\dots,e_k$, each with tail $u$ and head $v$,
such that $\sum_{i=1}^k\wt(e_i)=\wt(e)$.  Also, we would like to be able to
embed a pair of oppositely oriented edges between the same vertices so that they coincide
in the plane.  For these purposes then, we work in the more general category of
weighted directed {\em multi}-graphs by allowing $E$ to be a {\em multiset} of
edges in which an edge $e$ with endpoints $u$ and $v$ is represented as the set
$e=\{u,v\}$ with a pair of weights $\wt(e,(u,v))$ and $\wt(e,(v,u))$, at least
one of which is nonzero.  Each edge in the embedding is then represented by a double-headed arrow
with two weight labels (the label $\wt(e,(u,v))$ being placed next to the head
vertex, $v$).  Figure~\ref{fig:embedded graph} shows a pair of edges $e=\{u,v\}$
and $e'=\{u,v\}$ where $\wt(e,(u,v)))=2$, $\wt(e,(v,u)))=0$, $\wt(e',(u,v)))=3$,
and $\wt(e',(v,u)))=1$.  The top edge,~$e$, represents a single directed edge
$(u,v)$ of weight~$2$, and the bottom edge represents a pair of directed edges
of weights $3$ and $1$.  The two edges combine to represent a pair of directed
edges, $(u,v)$ of weight~$5$ and $(v,u)$ of weight~$1$.
\begin{figure}[ht] 
\begin{tikzpicture}[scale=1.0]
\SetVertexMath
\GraphInit[vstyle=Art]
\SetUpVertex[MinSize=3pt]
\SetVertexLabel
\tikzset{VertexStyle/.style = {%
shape = circle,
shading = ball,
ball color = black,
inner sep = 1.5pt
}}
\SetUpEdge[color=black]

\tikzset{EdgeStyle/.append style = {<->,>=mytip,semithick}}
\Vertex[LabelOut=true,Lpos=180,x=0,y=0]{u}
\Vertex[LabelOut=true,Lpos=0,x=3,y=0]{v}
\Edge[style={bend left = 20}](u)(v)
\Edge[style={bend left = 20}](v)(u)
\draw (0.4,0.5) node{0};
\draw (0.4,-0.5) node{1};
\draw (2.6,0.5) node{2};
\draw (2.6,-0.5) node{3};
\end{tikzpicture}
\caption{Edges for a planar embedding of a weighted directed graph.}\label{fig:embedded graph}
\end{figure}
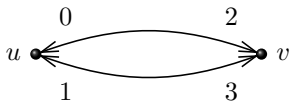

The rough idea of the construction of the weighted undirected graph~$\tG$ is to
overlay the embedded graph~$\Gamma$ with its dual, forgetting the orientation of
the edges and introducing new vertices where their edges cross.  Then remove $s$
and the vertex corresponding to the chosen face $f_s$, and remove their incident
edges.  In detail, the vertices of $\tG$ are
\[
V_{\tG}:=\{t_v:v\in V\setminus\{s\}\}\cup\{t_e: e\in
E\}\cup\{t_f:f\in F\setminus\{f_s\}\},
\]
where $F$ is the set of faces of $\Gamma$, including the unbounded face, and the
edges of~$\tG$ are
\[
E_{\tG}:=\{\{t_u,t_e\}: u\in V\setminus\{s\}, u\in e\in E\}\cup\{\{t_e,t_f\}:
e\in E, e\in f\in F\setminus\{f_s\}\}.
\]
The weight of each edge of the form $\{t_u,t_e\}$ with $e=\{u,v\}\in E$ is defined to be
$\wt(e,(u,v))$,  and the weight of each edge of the form $\{t_e,t_f\}$ with $f\in
F$ is defined to be $1$.  

Figure~\ref{fig:h-gamma} depicts a graph $\Gamma$ embedded in the plane (for
which the multiset $E$ is actually just a set).  The graph displayed in the
middle is the superposition of~$\Gamma$ with its dual, $\Gamma^{\perp}$.  The
unbounded face is chosen as~$f_s$.  For convenience, its corresponding vertex is omitted
from the middle graph, and its incident edges are only partially drawn.
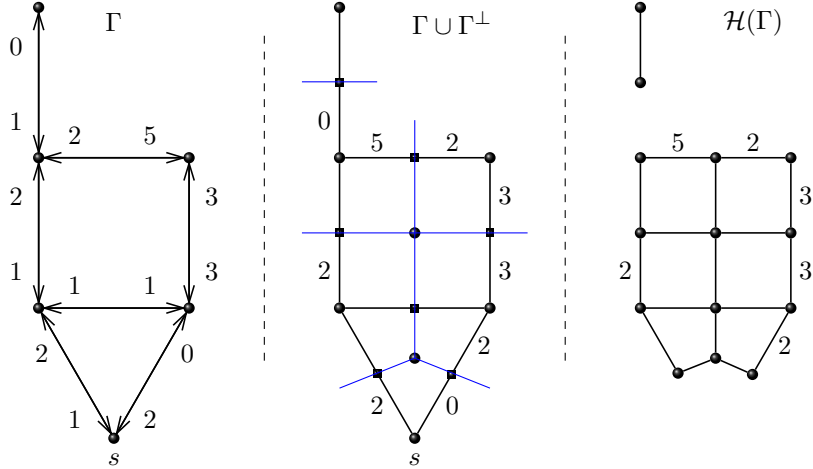
\begin{figure}[ht] 
\begin{tikzpicture}[scale=1.0]
\SetVertexMath
\GraphInit[vstyle=Art]
\SetUpVertex[MinSize=3pt]
\SetVertexLabel
\tikzset{VertexStyle/.style = {%
shape = circle,
shading = ball,
ball color = black,
inner sep = 1.5pt
}}
\SetUpEdge[color=black]

\tikzset{EdgeStyle/.append style = {->,>=mytip,semithick}}
\Vertex[NoLabel,x=0,y=4]{v_1}
\Vertex[NoLabel,x=0,y=2]{v_2}
\Vertex[NoLabel,x=0,y=0]{v_3}
\Vertex[LabelOut=true,Lpos=270,L=s,x=1,y=-1.732]{v_4}
\Vertex[NoLabel,x=2,y=2]{v_5}
\Vertex[NoLabel,x=2,y=0]{v_6}
\Edge[](v_1)(v_2)
\Edge[](v_2)(v_1)
\Edge[](v_2)(v_5)
\Edge[](v_5)(v_2)
\Edge[](v_2)(v_3)
\Edge[](v_3)(v_2)
\Edge[](v_3)(v_6)
\Edge[](v_6)(v_3)
\Edge[](v_5)(v_6)
\Edge[](v_6)(v_5)
\Edge[](v_3)(v_4)
\Edge[](v_4)(v_3)
\Edge[](v_6)(v_4)
\Edge[](v_4)(v_6)
\draw (-0.3,3.48) node{0};
\draw (-0.3,2.48) node{1};
\draw (-0.3,1.48) node{2};
\draw (-0.3,0.48) node{1};
\draw (1.48,0.3) node{1};
\draw (0.48,0.3) node{1};
\draw (2.3,1.48) node{3};
\draw (2.3,0.48) node{3};
\draw (1.48,2.3) node{5};
\draw (0.48,2.3) node{2};
\draw (0.04,-0.6) node{2};
\draw (0.48,-1.48) node{1};
\draw (1.48,-1.48) node{2};
\draw (1.98,-0.60) node{0};

\draw [dashed] (3.0,-0.7) -- (3.0,3.7);

\tikzset{EdgeStyle/.append style = {-}}
\Vertex[NoLabel,x=4,y=4]{v_1}
\Vertex[NoLabel,empty=true,x=4,y=3]{x03}
\Vertex[NoLabel,x=4,y=2]{v_2}
\Vertex[NoLabel,x=4,y=0]{v_3}
\Vertex[LabelOut=true,Lpos=270,L=s,x=5,y=-1.732]{v_4}
\Vertex[NoLabel,x=6,y=2]{v_5}
\Vertex[NoLabel,x=6,y=0]{v_6}
\tikzset{VertexStyle/.append style={shape=rectangle}}
\Vertex[NoLabel,x=4,y=3]{x03}
\Vertex[NoLabel,x=4,y=1]{x01}
\Vertex[NoLabel,x=5,y=2]{x12}
\Vertex[NoLabel,x=6,y=1]{x21}
\Vertex[NoLabel,x=5,y=0]{x10}
\Vertex[NoLabel,x=4.5,y=-0.866]{xneg1}
\Vertex[NoLabel,x=5.490,y=-0.88335]{xneg2}
\tikzset{VertexStyle/.append style={shape=circle}}

\Edge[](v_1)(x03)
\Edge[](v_2)(x03)
\Edge[](v_2)(v_3)
\Edge[](v_2)(v_5)
\Edge[](v_3)(v_6)
\Edge[](v_3)(v_4)
\Edge[](v_4)(xneg2)
\Edge[](v_5)(v_6)
\Edge[](v_6)(xneg2)
\draw (3.8,2.5) node{0};
\draw (3.8,0.5) node{2};
\draw (6.2,1.5) node{3};
\draw (6.2,0.5) node{3};
\draw (4.5,2.2) node{5};
\draw (5.5,2.2) node{2};
\draw (4.50,-1.3) node{2};
\draw (5.92,-0.5) node{2};
\draw (5.5,-1.3) node{0};

\Vertex[NoLabel,x=5,y=1]{mt}
\Vertex[NoLabel,x=5,y=-0.667]{mb}
\draw [blue] (3.5,3.01) -- (4.5,3.01);
\draw [blue] (5,1) -- (5,2.5);
\draw [blue] (5,1) -- (5,-0.666);
\draw [blue] (5,1) -- (3.5,1);
\draw [blue] (5,1) -- (6.5,1);
\draw [blue] (5,-0.666) -- (4,-1.0654);
\draw [blue] (5,-0.666) -- (6,-1.0654);

\draw [dashed] (7.0,-0.7) -- (7.0,3.7);

\Vertex[NoLabel,x=8,y=4]{v_1}
\Vertex[NoLabel,x=8,y=3]{x03}
\Vertex[NoLabel,x=8,y=2]{v_2}
\Vertex[NoLabel,x=8,y=1]{x01}
\Vertex[NoLabel,x=8,y=0]{v_3}
\Vertex[NoLabel,x=8.5,y=-0.866]{xneg1}
\Vertex[NoLabel,x=9,y=2]{x12}
\Vertex[NoLabel,x=9,y=0]{x10}
\Vertex[NoLabel,x=9.490,y=-0.88335]{xneg2}
\Vertex[NoLabel,x=10,y=2]{v_5}
\Vertex[NoLabel,x=10,y=1]{x21}
\Vertex[NoLabel,x=10,y=0]{v_6}
\Vertex[NoLabel,x=9,y=1]{mt}
\Vertex[NoLabel,x=9,y=-0.667]{mb}
\Edge[](v_1)(x03)
\Edge[](v_2)(x01)
\Edge[](v_2)(x12)
\Edge[](v_3)(x01)
\Edge[](v_3)(x10)
\Edge[](v_3)(xneg1)
\Edge[](v_5)(x12)
\Edge[](v_5)(x21)
\Edge[](v_6)(x21)
\Edge[](v_6)(x10)
\Edge[](v_6)(xneg2)
\Edge[](mt)(x12)
\Edge[](mt)(x10)
\Edge[](mt)(x01)
\Edge[](mt)(x21)
\Edge[](mb)(x10)
\Edge[](mb)(xneg1)
\Edge[](mb)(xneg2)
\draw (7.8,0.5) node{2};
\draw (10.2,1.5) node{3};
\draw (10.2,0.5) node{3};
\draw (8.5,2.2) node{5};
\draw (9.5,2.2) node{2};
\draw (9.92,-0.5) node{2};

\draw (1,3.8) node(G){$\Gamma$};
\draw (5.5,3.8) node(G){$\Gamma\cup\Gamma^{\perp}$};
\draw (9.5,3.8) node(G){$\tG$};
\end{tikzpicture}
\caption{Construction of $\tG$. (Unlabeled edges have weight~$1$.)}\label{fig:h-gamma}
\end{figure}

A {\em perfect matching} of a weighted undirected graph is a subset of its edges such that
each vertex of the graph is incident with exactly one edge in the subset.  The {\em weight}
of a perfect matching is the product of the weights of its edges.  

We now describe the weight-preserving bijection between perfect matchings
of $\tG$ and directed spanning trees of $\Gamma$ rooted at $s$ due to
\cite{KPW}.  Let $T$ be a directed spanning tree of $\Gamma$ rooted at $s$, and
let $\widetilde{T}$ be the corresponding directed spanning tree of
$\Gamma^{\perp}$, the dual of $\Gamma$, rooted at~$f_s$.  (The tree
$\widetilde{T}$ is obtained by properly orienting the edges of~$\Gamma^{\perp}$ that do not cross
edges of $T$ in  $\Gamma\cup\Gamma^{\perp}$.) The perfect matching
of $\tG$ corresponding to $T$ consists of the following:
\begin{enumerate}
  \item an edge $\{t_u,t_e\}$ of weight $\wt(e)$ for each $e=(u,v)\in T$;
  \item an edge $\{t_f,t_e\}$ of weight $1$ for each $\tilde{e}=(f,f')\in
    \widetilde{T}$, where $e$ is the edge in $\Gamma$ that crossed by $\tilde{e}$.
\end{enumerate}
See Figure~\ref{fig:h-gamma2} for an example continuing the example from Figure~\ref{fig:h-gamma}.
\begin{figure}[ht] 
\begin{tikzpicture}[scale=1.0]
\SetVertexMath
\GraphInit[vstyle=Art]
\SetUpVertex[MinSize=3pt]
\SetVertexLabel
\tikzset{VertexStyle/.style = {%
shape = circle,
shading = ball,
ball color = black,
inner sep = 1.9pt
}}
\SetUpEdge[color=black]

\tikzset{EdgeStyle/.append style = {->,>=mytip,ultra thick}}
\Vertex[NoLabel,x=4,y=4]{v_1}
\Vertex[NoLabel,empty=true,x=4,y=3]{x03}
\Vertex[NoLabel,x=4,y=2]{v_2}
\Vertex[NoLabel,x=4,y=0]{v_3}
\Vertex[LabelOut=true,Lpos=270,L=s,x=5,y=-1.732]{v_4}
\Vertex[NoLabel,x=6,y=2]{v_5}
\Vertex[NoLabel,x=6,y=0]{v_6}
\tikzset{VertexStyle/.append style={shape=rectangle}}
\Vertex[NoLabel,x=4,y=3]{x03}
\Vertex[NoLabel,x=4,y=1]{x01}
\Vertex[NoLabel,x=5,y=2]{x12}
\Vertex[NoLabel,x=6,y=1]{x21}
\Vertex[NoLabel,x=5,y=0]{x10}
\Vertex[NoLabel,x=4.5,y=-0.866]{xneg1}
\Vertex[NoLabel,x=5.490,y=-0.88335]{xneg2}
\tikzset{VertexStyle/.append style={shape=circle}}
\Vertex[NoLabel,x=5,y=1]{mt}
\Vertex[NoLabel,x=5,y=-0.667]{mb}
\Vertex[NoLabel,empty=true,x=3.2,y=1]{outleft}

\Edge[](v_1)(v_2)
\Edge[](v_2)(v_5)
\Edge[](v_5)(v_6)
\Edge[](v_6)(v_4)
\Edge[](v_3)(v_4)
\Edge[style={color=blue}](mb)(mt)
\Edge[style={color=blue}](mt)(outleft)
\draw (6.2,1.2) node{3};
\draw (4.8,2.23) node{5};
\draw (5.83,-0.74) node{2};

\draw [blue] (3.5,3.01) -- (4.5,3.01);
\draw [blue] (5,1) -- (5,2.5);
\draw [blue] (5,1) -- (5,-0.666);
\draw [blue] (5,1) -- (3.5,1);
\draw [blue] (5,1) -- (6.5,1);
\draw [blue] (5,-0.666) -- (4,-1.0654);
\draw [blue] (5,-0.666) -- (6,-1.0654);

\draw [dashed] (7.5,-0.7) -- (7.5,3.7);

\tikzset{VertexStyle/.style = {%
shape = circle,
shading = ball,
ball color = black,
inner sep = 1.5pt
}}
\Vertex[NoLabel,x=9,y=4]{v_1}
\Vertex[NoLabel,x=9,y=3]{x03}
\Vertex[NoLabel,x=9,y=2]{v_2}
\Vertex[NoLabel,x=9,y=1]{x01}
\Vertex[NoLabel,x=9,y=0]{v_3}
\Vertex[NoLabel,x=9.5,y=-0.866]{xneg1}
\Vertex[NoLabel,x=10,y=2]{x12}
\Vertex[NoLabel,x=10,y=0]{x10}
\Vertex[NoLabel,x=10.490,y=-0.88335]{xneg2}
\Vertex[NoLabel,x=11,y=2]{v_5}
\Vertex[NoLabel,x=11,y=1]{x21}
\Vertex[NoLabel,x=11,y=0]{v_6}
\Vertex[NoLabel,x=10,y=1]{mt}
\Vertex[NoLabel,x=10,y=-0.667]{mb}
\draw[thick] (v_1) -- (x03);
\draw[thick] (v_2) -- (x12);
\draw[thick] (v_5) -- (x21);
\draw[thick] (v_6) -- (xneg2);
\draw[thick] (v_3) -- (xneg1);
\draw[thick] (mb) -- (x10);
\draw[thick] (mt) -- (x01);
\draw (11.2,1.5) node{3};
\draw (9.5,2.2) node{5};
\draw (10.92,-0.5) node{2};
\draw (10.6,3.8) node(G){$\tG$};
\end{tikzpicture}
\caption{A spanning tree of $\Gamma$ determines a dual spanning tree for
$\Gamma^{\perp}$ and a perfect matching for $\tG$. (See
Figure~\ref{fig:h-gamma}.  Unlabeled edges have weight $1$.)}\label{fig:h-gamma2}
\end{figure}
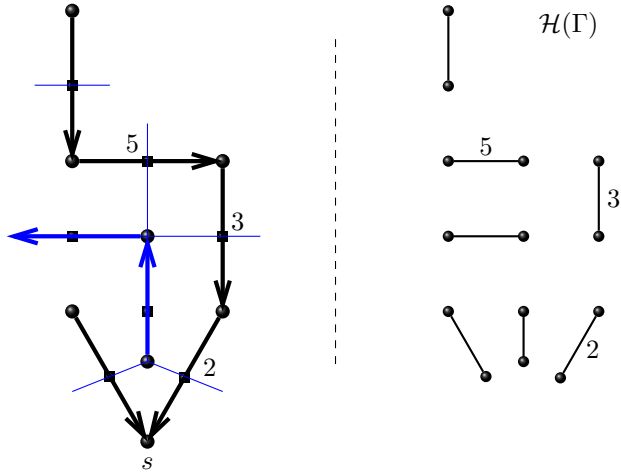

As discussed in~\cite{KPW}, although $\mathcal{H}(\Gamma)$ depends on the
embedding of $\Gamma$ and on the choice of $f_s$, the number of spanning trees
of $\Gamma$ rooted at $s$ (and hence, the number of perfect matchings of
$\mathcal{H}(\Gamma)$), counted according to weight, does not change.  In what
follows, we will always choose $f_s$ to be the unbounded face.

\section{Symmetric recurrents on the sandpile grid
graph}\label{section:symmetric recurrents} The {\em ordinary $m\times n$ grid
graph} is the undirected graph $\Gamma_{m\times n}$ with vertices $[m]\times[n]$
and edges $\{(i,j),(i',j')\}$ such that $|i-i'|+|j-j'|=1$.  The $m\times n$ {\em
sandpile grid graph}, $\sg_{m\times n}$, is formed from $\Gamma_{m\times n}$ by
adding a (disjoint) sink vertex, $s$, then edges incident to $s$ so that every
non-sink vertex of the resulting graph has degree~$4$.  For instance, each of the four
corners of the sandpile grid graph shares an edge of weight~$2$ with~$s$ in the
case where $m\geq2$ and $n\geq 2$, as on the left in Figure~\ref{fig:sandpile
grid graph}.  

We embed $\Gamma_{m\times n}$ in the plane as the standard grid with vertices
arranged as in a matrix, with $(1,1)$ in the upper left and $(m,n)$ in the lower
right.  We embed $\sg_{m\times n}$ similarly, but usually identify the sink
vertex, $s$, with the unbounded face of $\Gamma_{m\times n}$ for convenience in
drawing, as on the left-hand side in Figure~\ref{fig:sandpile grid graph}.  The
edges leading to the sink are sometimes entirely omitted from the drawing, as in
Figure~\ref{fig:8x6}.
\begin{figure}[ht] 
\begin{tikzpicture}[scale=0.7]
\draw[style=very thick] (1,1) grid (5,4);
\draw (1,1) -- (0.2,0.2);
\draw (1,4) -- (0.2,4.8);
\draw (5,1) -- (5.8,0.2);
\draw (5,4) -- (5.8,4.8);
\foreach \i in {2,3,4}{
\draw (\i,1) -- (\i, 0.2);
\draw (\i,4) -- (\i, 4.8);
}
\foreach \i in {2,3}{
\draw (1,\i) -- (0.2,\i);
\draw (5,\i) -- (5.8,\i);
}
\draw[fill=white, draw=none] (0.6,0.6) circle [radius=2.8mm];
\draw (0.6,0.6) node{$2$};
\draw[fill=white, draw=none] (0.6,4.4) circle [radius=2.8mm];
\draw (0.6,4.4) node{$2$};
\draw[fill=white, draw=none] (5.4,0.6) circle [radius=2.8mm];
\draw (5.4,0.6) node{$2$};
\draw[fill=white, draw=none] (5.4,4.4) circle [radius=2.8mm];
\draw (5.4,4.4) node{$2$};
\node at (3.3,-0.4){$\sg_{4\times 5}$};
\def\i{10}
\draw[style=very thick] (\i+3,1) -- (\i+3,4);
\draw (\i,2.5) .. controls +(280:20pt) and +(180:20pt) .. (\i+3,1);
\draw (\i,2.5) .. controls +(330:20pt) and +(180:20pt) .. (\i+3,2);
\draw (\i,2.5) .. controls +(30:20pt) and +(180:20pt) .. (\i+3,3);
\draw (\i,2.5) .. controls +(80:20pt) and +(180:20pt) .. (\i+3,4);

\draw[fill=white, draw=none] (\i+1,3.4) circle [radius=2.8mm];
\draw (\i+1,3.4) node{$3$};
\draw[fill=white, draw=none] (\i+1.7,2.9) circle [radius=2.8mm];
\draw (\i+1.7,2.9) node{$2$};
\draw[fill=white, draw=none] (\i+1.7,2.1) circle [radius=2.8mm];
\draw (\i+1.7,2.1) node{$2$};
\draw[fill=white, draw=none] (\i+1.1,1.53) circle [radius=2.8mm];
\draw (\i+1.1,1.53) node{$3$};

\draw[fill] (\i,2.5) circle [radius=2pt];
\node at (\i-0.5,2.5){$s$};
\node at (\i+1.5,-0.4){$\sg_{4\times 1}$};
\end{tikzpicture}
\caption{Two sandpile grid graphs. (The sink for $\sg_{4\times 5}$ is not
drawn.)}\label{fig:sandpile grid graph}
\end{figure}
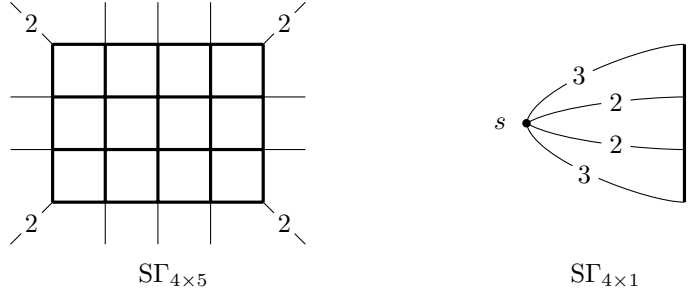

In this section, {\em symmetric recurrent} will always refer to a recurrent
element on $\sg_{m\times n}$ with horizontal and vertical symmetry, i.e., an
element of $\sand(\sg_{m\times n})^G$ where~$G$ is the Klein $4$-group,
\[
G=\langle\sigma,\tau:\sigma^2=\tau^2=1, \sigma\tau=\tau\sigma\rangle,
\]
acting on $\sg_{m\times n}$ by
\[
\sigma(i,j)=(i,n-j+1),\quad
\tau(i,j)=(m-i+1,j),\quad\mbox{and $\sigma(s)=\tau(s)=s$}.
\]

Our main goal in this section is to study the symmetric recurrent configurations
on the sandpile grid graph.  After collecting some basic facts about certain
tridiagonal matrices, we divide the study into three cases: even$\times$even-,
even$\times$odd-, and odd$\times$odd-dimensional grids.  In each case we
provide a formula for the number of symmetric recurrents using Chebyshev
polynomials and show how these configurations are related to domino tilings of
various types of checkerboards.

\subsection{Some tridiagonal matrices.}\label{subsection:tridiagonal}
Recall that {\em Chebyshev polynomials of the first kind} are defined by the
recurrence
\begin{align}\label{eqn:1st}
\nonumber
T_0(x)&=1\\
T_1(x)&=x\\
\nonumber
T_{j}(x)&=2x\,T_{j-1}(x)-T_{j-2}(x)\quad\text{for $j\geq2$},
\end{align}
and {\em Chebyshev polynomials of the second kind} are defined by
\begin{align}\label{eqn:2nd}
\nonumber
U_0(x)&=1\\
U_1(x)&=2x\\
\nonumber
U_{j}(x)&=2x\,U_{j-1}(x)-U_{j-2}(x)\quad\text{for $j\geq2$}.
\end{align}
Two references are~\cite{MH} and~\cite{wiki-chebyshev}.

It follows from the recurrences that these polynomials may be expressed as
determinants of $j\times j$ tridiagonal matrices:
\[
T_j(x) = \det
\resizebox{.3\textwidth}{!}{
$
\begin{bmatrix}
  x&1&&&&\\
  1&2x&1&&&\\
  &1&2x&1&&\\
  &&&\ddots&&\\
  &&&1&2x&1\\
  &&&&1&2x
\end{bmatrix},
$
}
\quad
U_j(x) = \det
\resizebox{.3\textwidth}{!}{
$
\begin{bmatrix}
  2x&1&&&&\\
  1&2x&1&&&\\
  &1&2x&1&&\\
  &&&\ddots&&\\
  &&&1&2x&1\\
  &&&&1&2x
\end{bmatrix},
$
}
\]
and, hence, $T_j(-x)=(-1)^j\,T_j(x)$ and $U_j(-x)=(-1)^j\,U_j(x)$.  

We have the well-known factorizations:
\begin{align}
T_j(x)&=2^{j-1}\prod_{k=1}^j\left(x-\cos\left(\frac{(2k-1)\pi}{2j} \right)
\right)\label{T-factorization}\\[5pt]
U_j(x)&=2^{j}\prod_{k=1}^j\left(x-\cos\left(\frac{k\pi}{j+1} \right) \right).
\label{U-factorization}
\end{align}
We will also use the following well-known identities:
\begin{align}
T_{2j}(x)&=T_j(2x^2-1)=(-1)^j\,T_j(1-2x^2)\label{eqn:half-angle}\\
2\,T_j(x)&= U_j(x)-U_{j-2}(x).\label{eqn:sum formula}
\end{align}

Corollary~\ref{cor:nsr} will be used to count the symmetric recurrents on
sandpile grid graphs. The form of the determinant that arises is treated by the
following.

\begin{lemma}\label{lemma:tridiagonal} Let $m$ and $n$ be positive integers.
  Let $A$, $B$, and $C$ be $n\times n$ matrices over the complex numbers, and
  let $I_n$ be the $n\times n$ identity matrix.  Define the $mn\times mn$
  tridiagonal block matrix
 \[
 D(m) =
 \left[
 \begin{array}{ccccc}
   A&-I_n&&&\\
   -I_n&A&-I_n&&\\
   &&\ddots&&\\
   &&-I_n&A&-I_n\\
   &&&-C&B
 \end{array}
 \right],
 \]
where the super- and sub-diagonal blocks are all $-I_n$ except for the
 one displayed block consisting of $-C$ and all omitted entries in the matrix are zero.  Take $D(1) = B$. Then
 \[
 \det D(m)=(-1)^n\det(T),
 \]
 where 
 \[
  T = -B\,U_{m-1}\left(\frac{1}{2}A\right)+C\,U_{m-2}\left(\frac{1}{2}A\right),
 \]
 letting $U_{-1}(x):=0$.
\end{lemma}
\begin{proof} The case $m=1$ is immediate.  For $m>1$, Theorem~2 of
  \cite{Molinari} gives a formula for calculating the determinant of a general
  tridiagonal block matrix.  In our case, it says 
\begin{equation}
  \det D(m) = (-1)^n\det E_{\boldsymbol{t}},
  \label{eqn:Et}
\end{equation}
where $E_{\boldsymbol{t}}$ is the top-left block of size
$n\times n$ of the matrix
\[
E:=
\begin{bmatrix}
  -B&C\\
  I_n&0
\end{bmatrix}
\begin{bmatrix}
  A&-I_n\\
  I_n&0
\end{bmatrix}^{m-2}
\begin{bmatrix}
  A&I_n\\
  I_n&0
\end{bmatrix}.
\]
Set $S_0=I_n$, and for all positive integers $j$, define 
\[
S_j = \left( \begin{bmatrix}A& -I_n \\ I_n &
  0\end{bmatrix}^{j-1}\begin{bmatrix}A& I_n \\ I_n & 0\end{bmatrix}
    \right)_{\boldsymbol{t}}
\]
and 
\[ 
S'_j = \left (\begin{bmatrix}A& -I_n \\ I_n &
  0\end{bmatrix}^{j-1}\begin{bmatrix}A& I_n \\ I_n & 0\end{bmatrix}
    \right)_{\boldsymbol{b}}, 
\]
where the subscripts $\boldsymbol{t}$ and $\boldsymbol{b}$ denote taking the
top-left and bottom-left blocks of size $n \times n$, respectively.
It follows that 
\begin{equation}\label{eqn:srecurrence}
S_0 = I_n ,\quad S_1= A,\quad\text{and $\quad S_j= A\,S_{j-1} - S_{j-2}$ for $j\geq
2$},
\end{equation}
and 
\[
S'_j = S_{j-1}\;\,\text{for all}\;\, j \ge 1.
\]
By \eqref{eqn:2nd} and \eqref{eqn:srecurrence}, $S_{j}= U_{j} (\frac{1}{2}A)$.
Hence,
\[
E_{\boldsymbol{t}}= -B\,S_{m-1}+C\,S'_{m-1} =
-B\,U_{m-1}\left(\frac{1}{2}A\right)+C\,U_{m-2}\left(\frac{1}{2}A\right),
\]
as required.
\end{proof}
\subsection{Symmetric recurrents on a $2m\times 2n$ sandpile grid
graph.}\label{subsection:symmetric recurrents on evenxeven grid} A {\em
checkerboard} is a rectangular array of squares.  A domino is a $1\times 2$ or
$2\times 1$ array of squares and, thus, covers exactly two adjacent squares of
the checkerboard. A {\em domino tiling of the checkerboard} consists of placing
non-overlapping dominos on the checkerboard, covering every square.  As is usually
done, and exhibited in Figure~\ref{fig:perfect matching}, we identify domino
tilings of an $m\times n$ checkerboard with perfect matchings of
$\Gamma_{m\times n}$.  Figure~\ref{fig:checker4x4} exhibits the $36$ domino
tilings of a $4\times4$ checkerboard.
\begin{figure}[ht] 
\begin{tikzpicture}[scale=0.7]
\newcommand{\vertRect}[2]{\draw[fill=gray!20,rounded corners=0.4mm](#1+0.2,#2+0.2) rectangle (#1+0.8,#2+1.8);}
\newcommand{\horizRect}[2]{\draw[fill=gray!20,rounded corners=0.4mm] (#1+0.2,#2+0.2) rectangle (#1+1.8,#2+0.8);}

\draw (0,0) grid (4,3);
\horizRect{0}{0}; 
\vertRect{2}{0};
\vertRect{3}{0};
\vertRect{0}{1};
\vertRect{1}{1};
\horizRect{2}{2}; 

\foreach \i in {0,1,2,3}{
  \foreach \j in {0,1,2}{
  \draw[fill] (\i+0.5,\j+0.5) circle [radius=2pt];
  }
}
\foreach \j in {0,1,2}{
\draw[style=very thin,color=gray!90] (0.5,\j+0.5) -- (3.5,\j+0.5);
}
\foreach \i in {0,1,2,3}{
\draw[very thin,color=gray!90] (\i+0.5,0.5) -- (\i+0.5,2.5);
}
\draw[ultra thick] (0.5,0.5) -- (1.5,0.5);
\draw[ultra thick] (2.5,0.5) -- (2.5,1.5);
\draw[ultra thick] (3.5,0.5) -- (3.5,1.5);
\draw[ultra thick] (0.5,1.5) -- (0.5,2.5);
\draw[ultra thick] (1.5,1.5) -- (1.5,2.5);
\draw[ultra thick] (2.5,2.5) -- (3.5,2.5);
\end{tikzpicture}
\caption{Correspondence between a perfect matching of~$\Gamma_{3\times 4}$ and
a domino tiling of its corresponding checkerboard.}\label{fig:perfect matching}
\end{figure}
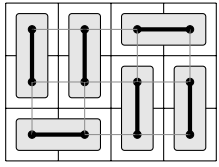

\begin{figure}[ht] 
\begin{tikzpicture}[scale=0.3]
\newcommand{\vertRect}[3]{\draw[fill=#3,rounded corners=0.4mm](#1+0.2,#2+0.2) rectangle (#1+0.8,#2+1.8);}
\newcommand{\horizRect}[3]{\draw[fill=#3,rounded corners=0.4mm] (#1+0.2,#2+0.2) rectangle (#1+1.8,#2+0.8);}
\foreach \i in {0,...,5}{
  \foreach \j in {0,...,5} { 
  \draw (5*\i,5*\j) grid (5*\i+4,5*\j+4);
  }
}

\def\j{25}
\horizRect{0}{\j+3}{gray} \horizRect{2}{\j+3}{gray};
\horizRect{0}{\j+2}{gray} \horizRect{2}{\j+2}{blue};
\horizRect{0}{\j+1}{gray} \horizRect{2}{\j+1}{gray};
\horizRect{0}{\j}{gray} \horizRect{2}{\j}{blue};

\horizRect{5}{\j+3}{gray} \vertRect{7}{\j+2}{gray};
\horizRect{5}{\j+2}{gray} \vertRect{8}{\j+2}{gray};
\horizRect{5}{\j+1}{gray} \horizRect{7}{\j+1}{gray};
\horizRect{5}{\j}{gray} \horizRect{7}{\j}{blue};

\horizRect{10}{\j+3}{gray} \vertRect{12}{\j+2}{gray};
\horizRect{10}{\j+2}{gray} \vertRect{13}{\j+2}{gray};
\horizRect{10}{\j+1}{gray} \vertRect{13}{\j}{gray};
\horizRect{10}{\j}{gray} \vertRect{12}{\j}{gray};

\horizRect{15}{\j+3}{gray} \horizRect{17}{\j+3}{gray};
\horizRect{15}{\j+2}{gray} \vertRect{17}{\j+1}{gray};
\horizRect{15}{\j+1}{gray} \vertRect{18}{\j+1}{gray};
\horizRect{15}{\j}{gray} \horizRect{17}{\j}{blue};

\horizRect{20}{\j+3}{gray} \horizRect{22}{\j+3}{gray};
\horizRect{20}{\j+2}{gray} \horizRect{22}{\j+2}{blue};
\horizRect{20}{\j+1}{gray} \vertRect{22}{\j}{gray};
\horizRect{20}{\j}{gray} \vertRect{23}{\j}{gray};

\vertRect{25}{\j+2}{gray} \horizRect{26}{\j+3}{gray};
\horizRect{26}{\j+2}{gray} \vertRect{28}{\j+2}{gray};
\horizRect{25}{\j}{gray} \horizRect{27}{\j}{blue};
\horizRect{25}{\j+1}{gray} \horizRect{27}{\j+1}{gray};

\def\j{20}
\vertRect{0}{\j+2}{gray} \horizRect{2}{\j+3}{gray};
\vertRect{1}{\j+2}{gray} \horizRect{2}{\j+2}{blue};
\horizRect{0}{\j+1}{gray} \horizRect{2}{\j+1}{gray};
\horizRect{0}{\j}{gray} \horizRect{2}{\j}{blue};

\vertRect{5}{\j+2}{gray} \vertRect{7}{\j+2}{gray};
\vertRect{6}{\j+2}{gray} \vertRect{8}{\j+2}{gray};
\horizRect{5}{\j+1}{gray} \horizRect{7}{\j+1}{gray};
\horizRect{5}{\j}{gray} \horizRect{7}{\j}{blue};

\vertRect{10}{\j+2}{gray} \vertRect{12}{\j+2}{gray};
\vertRect{11}{\j+2}{gray} \vertRect{13}{\j+2}{gray};
\horizRect{10}{\j+1}{gray} \vertRect{13}{\j}{gray};
\horizRect{10}{\j}{gray} \vertRect{12}{\j}{gray};

\vertRect{15}{\j+2}{gray} \horizRect{17}{\j+3}{gray};
\vertRect{16}{\j+2}{gray} \vertRect{17}{\j+1}{gray};
\horizRect{15}{\j+1}{gray} \vertRect{18}{\j+1}{gray};
\horizRect{15}{\j}{gray} \horizRect{17}{\j}{blue};

\vertRect{20}{\j+2}{gray} \horizRect{22}{\j+3}{gray};
\vertRect{21}{\j+2}{gray} \horizRect{22}{\j+2}{blue};
\horizRect{20}{\j+1}{gray} \vertRect{22}{\j}{gray};
\horizRect{20}{\j}{gray} \vertRect{23}{\j}{gray};

\vertRect{25}{\j+2}{gray} \horizRect{26}{\j+3}{gray};
\horizRect{26}{\j+2}{gray} \vertRect{28}{\j+2}{gray};
\horizRect{25}{\j}{gray} \vertRect{27}{\j}{gray};
\horizRect{25}{\j+1}{gray} \vertRect{28}{\j}{gray};

\def\j{15}
\vertRect{0}{\j+2}{gray} \horizRect{2}{\j+3}{gray};
\vertRect{1}{\j+2}{gray} \horizRect{2}{\j+2}{blue};
\vertRect{0}{\j}{gray} \horizRect{2}{\j+1}{gray};
\vertRect{1}{\j}{gray} \horizRect{2}{\j}{blue};

\vertRect{5}{\j+2}{gray} \vertRect{7}{\j+2}{gray};
\vertRect{6}{\j+2}{gray} \vertRect{8}{\j+2}{gray};
\vertRect{5}{\j}{gray} \horizRect{7}{\j+1}{gray};
\vertRect{6}{\j}{gray} \horizRect{7}{\j}{blue};

\vertRect{10}{\j+2}{gray} \vertRect{12}{\j+2}{gray};
\vertRect{11}{\j+2}{gray} \vertRect{13}{\j+2}{gray};
\vertRect{10}{\j}{gray} \vertRect{13}{\j}{gray};
\vertRect{11}{\j}{gray} \vertRect{12}{\j}{gray};

\vertRect{15}{\j+2}{gray} \horizRect{17}{\j+3}{gray};
\vertRect{16}{\j+2}{gray} \vertRect{17}{\j+1}{gray};
\vertRect{15}{\j}{gray} \vertRect{18}{\j+1}{gray};
\vertRect{16}{\j}{gray} \horizRect{17}{\j}{blue};

\vertRect{20}{\j+2}{gray} \horizRect{22}{\j+3}{gray};
\vertRect{21}{\j+2}{gray} \horizRect{22}{\j+2}{blue};
\vertRect{20}{\j}{gray} \vertRect{22}{\j}{gray};
\vertRect{21}{\j}{gray} \vertRect{23}{\j}{gray};

\vertRect{25}{\j+2}{gray} \horizRect{26}{\j+3}{gray};
\horizRect{26}{\j+2}{gray} \vertRect{28}{\j+2}{gray};
\vertRect{25}{\j}{gray} \horizRect{27}{\j}{blue};
\vertRect{26}{\j}{gray} \horizRect{27}{\j+1}{gray};

\def\j{10}
\horizRect{0}{\j+3}{gray} \horizRect{2}{\j+3}{gray};
\vertRect{0}{\j+1}{gray} \horizRect{2}{\j+2}{blue};
\vertRect{1}{\j+1}{gray} \horizRect{2}{\j+1}{gray};
\horizRect{0}{\j}{gray} \horizRect{2}{\j}{blue};

\horizRect{5}{\j+3}{gray} \vertRect{7}{\j+2}{gray};
\vertRect{5}{\j+1}{gray} \vertRect{8}{\j+2}{gray};
\vertRect{6}{\j+1}{gray} \horizRect{7}{\j+1}{gray};
\horizRect{5}{\j}{gray} \horizRect{7}{\j}{blue};

\horizRect{10}{\j+3}{gray} \vertRect{12}{\j+2}{gray};
\vertRect{10}{\j+1}{gray} \vertRect{13}{\j+2}{gray};
\vertRect{11}{\j+1}{gray} \vertRect{13}{\j}{gray};
\horizRect{10}{\j}{gray} \vertRect{12}{\j}{gray};

\horizRect{15}{\j+3}{gray} \horizRect{17}{\j+3}{gray};
\vertRect{15}{\j+1}{gray} \vertRect{17}{\j+1}{gray};
\vertRect{16}{\j+1}{gray} \vertRect{18}{\j+1}{gray};
\horizRect{15}{\j}{gray} \horizRect{17}{\j}{blue};

\horizRect{20}{\j+3}{gray} \horizRect{22}{\j+3}{gray};
\vertRect{20}{\j+1}{gray} \horizRect{22}{\j+2}{blue};
\vertRect{21}{\j+1}{gray} \vertRect{22}{\j}{gray};
\horizRect{20}{\j}{gray} \vertRect{23}{\j}{gray};

\horizRect{25}{\j+3}{gray} \horizRect{27}{\j+3}{gray};
\horizRect{25}{\j+2}{gray} \horizRect{27}{\j+2}{blue};
\vertRect{25}{\j}{gray} \horizRect{26}{\j}{gray};
\vertRect{28}{\j}{gray} \horizRect{26}{\j+1}{gray};

\def\j{5}
\horizRect{0}{\j+3}{gray} \horizRect{2}{\j+3}{gray};
\horizRect{0}{\j+2}{gray} \horizRect{2}{\j+2}{blue};
\vertRect{0}{\j}{gray} \horizRect{2}{\j+1}{gray};
\vertRect{1}{\j}{gray} \horizRect{2}{\j}{blue};

\horizRect{5}{\j+3}{gray} \vertRect{7}{\j+2}{gray};
\horizRect{5}{\j+2}{gray} \vertRect{8}{\j+2}{gray};
\vertRect{5}{\j}{gray} \horizRect{7}{\j+1}{gray};
\vertRect{6}{\j}{gray} \horizRect{7}{\j}{blue};

\horizRect{10}{\j+3}{gray} \vertRect{12}{\j+2}{gray};
\horizRect{10}{\j+2}{gray} \vertRect{13}{\j+2}{gray};
\vertRect{10}{\j}{gray} \vertRect{13}{\j}{gray};
\vertRect{11}{\j}{gray} \vertRect{12}{\j}{gray};

\horizRect{15}{\j+3}{gray} \horizRect{17}{\j+3}{gray};
\horizRect{15}{\j+2}{gray} \vertRect{17}{\j+1}{gray};
\vertRect{15}{\j}{gray} \vertRect{18}{\j+1}{gray};
\vertRect{16}{\j}{gray} \horizRect{17}{\j}{blue};

\horizRect{20}{\j+3}{gray} \horizRect{22}{\j+3}{gray};
\horizRect{20}{\j+2}{gray} \horizRect{22}{\j+2}{blue};
\vertRect{20}{\j}{gray} \vertRect{22}{\j}{gray};
\vertRect{21}{\j}{gray} \vertRect{23}{\j}{gray};

\horizRect{25}{\j+3}{gray} \vertRect{27}{\j+2}{gray};
\horizRect{25}{\j+2}{gray} \vertRect{28}{\j+2}{gray};
\vertRect{25}{\j}{gray} \horizRect{26}{\j}{gray};
\vertRect{28}{\j}{gray} \horizRect{26}{\j+1}{gray};

\def\j{0}
\vertRect{0}{\j}{gray} \vertRect{0}{\j+2}{gray};
\vertRect{3}{\j}{gray} \vertRect{3}{\j+2}{gray};
\horizRect{1}{\j}{gray} \vertRect{1}{\j+2}{gray};
\horizRect{1}{\j+1}{gray} \vertRect{2}{\j+2}{gray};

\vertRect{5}{\j}{gray} \vertRect{5}{\j+2}{gray};
\vertRect{8}{\j}{gray} \vertRect{8}{\j+2}{gray};
\vertRect{6}{\j}{gray} \horizRect{6}{\j+2}{gray};
\vertRect{7}{\j}{gray} \horizRect{6}{\j+3}{gray};

\vertRect{10}{\j}{gray} \vertRect{10}{\j+2}{gray};
\vertRect{13}{\j}{gray} \vertRect{13}{\j+2}{gray};
\horizRect{11}{\j}{gray} \vertRect{11}{\j+1}{gray};
\vertRect{12}{\j+1}{gray} \horizRect{11}{\j+3}{gray};

\vertRect{15}{\j}{gray} \vertRect{15}{\j+2}{gray};
\vertRect{18}{\j}{gray} \vertRect{18}{\j+2}{gray};
\horizRect{16}{\j}{gray} \horizRect{16}{\j+1}{gray};
\horizRect{16}{\j+2}{gray} \horizRect{16}{\j+3}{gray};

\horizRect{20}{\j+3}{gray} \horizRect{22}{\j+3}{gray};
\vertRect{20}{\j+1}{gray} \horizRect{21}{\j+2}{gray};
\vertRect{23}{\j+1}{gray} \horizRect{21}{\j+1}{gray};
\horizRect{20}{\j}{gray} \horizRect{22}{\j}{blue};

\horizRect{27}{\j+3}{gray} \vertRect{25}{\j+2}{gray};
\horizRect{27}{\j+2}{blue} \vertRect{26}{\j+2}{gray};
\vertRect{25}{\j}{gray} \horizRect{26}{\j}{gray};
\vertRect{28}{\j}{gray} \horizRect{26}{\j+1}{gray};
\end{tikzpicture}
\caption{The $36$ domino tilings of a $4\times 4$
checkerboard. The blue dominos are assigned weight $2$ for the purposes of
Theorem~\ref{thm2}.}\label{fig:checker4x4}
\end{figure}
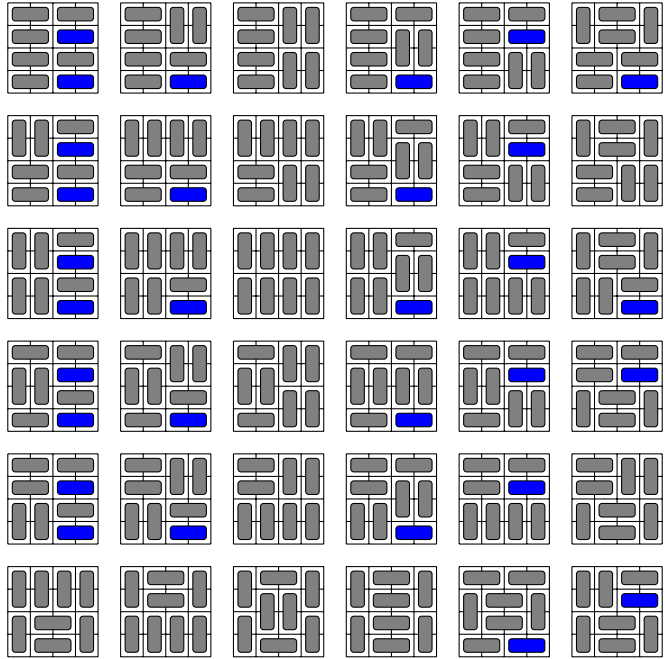

Part~(\ref{thm1-4}) of the following theorem is the well-known formula due to 
Kasteleyn {\cite{Kasteleyn}} and to Temperley and Fisher~\cite{Temperley} for
the number of domino tilings of a checkerboard.  We provide a new proof.
\begin{thm}\label{thm1} Let $U_j(x)$ denote the $j$-th Chebyshev polynomial of
  the second kind, and let
  \[
  \xi_{h,d}:=\cos\left(\frac{h\pi}{2d+1}\right),  
  \]
  for all integers $h$ and $d$.
  Then for all integers $m,n\geq 1$, the following are equal:
\begin{enumerate}
  \item\label{thm1-1} the number of symmetric recurrents on $\sg_{2m\times 2n}$;
  \item\label{thm1-2} the number of domino tilings of a $2m\times 2n$ checkerboard;
  \item\label{thm1-3} \ 
    \[
    (-1)^{mn}\prod_{h=1}^m U_{2n}(i\,\xi_{h,m});
    \]
  \item\label{thm1-4} \
    \[
    \prod_{h=1}^{m}\prod_{k=1}^{n}\left(4\,\xi_{h,m}^2+4\,\xi_{k,n}^2\right).
    \]
\end{enumerate}
\end{thm}
\begin{proof}  It may be helpful to read Example~\ref{example:main1} in parallel
  with this proof.

Let $A_n=(a_{h,k})$ be the $n\times n$ tridiagonal matrix with entries
\[
a_{h,k} = 
\begin{cases}
  \hfill 4 & \quad \text{if $h=k\neq n$},\\
  \hfill 3 & \quad \text{if $h=k=n$},\\
  -1& \quad \text{if $|h-k| = 1$},\\
  \hfill 0 & \quad \text{if $|h-k|\geq 2$}. 
\end{cases}
\]
In particular, $A_1=[3]$.  Take the vertices $[m]\times[n]$ as representatives
for the orbits of $G$ acting the non-sink vertices of $\sg_{2m\times 2n}$.
Ordering these representatives lexicographically, i.e., left-to-right then
top-to-bottom, the symmetrized reduced Laplacian~\eqref{eqn:srl} is given by the
$mn\times mn$ tridiagonal block matrix
\begin{equation}\label{eqn:deltaG}
\tD^G = \begin{bmatrix}
A_n  & -I_n  &         &         & \cdots  &         & 0 \\
-I_n  & A_n  & -I_n   &         &         &         & \\
       & \ddots & \ddots  & \ddots  &         &         & \vdots \\
       &        & -I_n   & A_n   & -I_n   &         & \\
\vdots &        &         & \ddots  & \ddots  & \ddots  & \\
       &        &         &         & -I_n & A_n & -I_n  \\
0      &        & \cdots  &         &         & -I_n   & B_n\\
\end{bmatrix}
\end{equation}
where $I_n$ is the $n\times n$ identity matrix and $B_n:=A_n-I_n$.  If $m=1$,
then $\tD^G:=B_n$.
\medskip

\noindent[{\bf(\ref{thm1-1}) $=$ (\ref{thm1-2})}]: The matrix $\tD^G$ is the
reduced Laplacian of a sandpile graph we now describe.  Let $D_{m\times n}$ be
the graph obtained from $\Gamma_{m\times n}$, the ordinary grid graph, by adding
(i) a sink vertex, $s'$, (ii) an edge of weight~$2$ from the vertex $(1,1)$
to~$s'$, and (iii) edges of weight $1$ from each of the other vertices along the
left and top sides to $s'$, i.e., $\{(h,1),s'\}$ for $1<h\leq m$ and
$\{(1,k),s'\}$ for $1<k\leq n$.  We embed $D_{m\times n}$ in the plane so that
the non-sink vertices form an ordinary grid, and the edge of weight $2$ is
represented by a pair of edges of weight $1$, forming a digon.  Then,
$\mathcal{H}(D_{m\times n})=\Gamma_{2m\times 2n}$ (see Figure~\ref{fig:4x3}).

Since $\tD^G=\tD_{D_{m\times n}}$, taking determinants shows that the number of
symmetric recurrents on $\sg_{2m\times 2n}$ is equal to the size of the sandpile
group of $D_{m\times n}$, and hence to the number of spanning trees of
$D_{m\times n}$ rooted at $s'$, counted according to weight.  These spanning
trees are, in turn, in bijection with the perfect matchings of the graph
$\mathcal{H}(D_{m\times n})=\Gamma_{2m\times 2n}$ obtained from the generalized
Temperley bijection of Section~\ref{section:Matchings and trees}.  Hence, the
numbers in parts~(\ref{thm1-1}) and~(\ref{thm1-2}) are equal.  
\medskip

\noindent [{\bf(\ref{thm1-1}) $=$ (\ref{thm1-3})}]: By Corollary~\ref{cor:nsr},
$\det\tD^G$ is the number of symmetric recurrents on $\sg_{2m\times 2n}$.  By
Lemma~\ref{lemma:tridiagonal},
\begin{equation}\label{eqn:thm1-det}
\det\tD^G = (-1)^n\det(T),
\end{equation}
where
\begin{align*}
T&= -B_n\,U_{m-1}\left(\frac{A_n}{2}\right)+
                      U_{m-2}\left(\frac{A_n}{2}\right)\\[5pt]
&= -(A_n-I_n)\,U_{m-1}\left(\frac{A_n}{2}\right)+U_{m-2}\left(\frac{A_n}{2}\right)\\[5pt]
       &= U_{m-1}\left(\frac{A_n}{2}\right)-\left(A_n\,U_{m-1}\left(\frac{A_n}{2}\right) -
      U_{m-2}\left(\frac{A_n}{2}\right)\right)\\[5pt]
       &=U_{m-1}\left(\frac{A_n}{2}\right) - U_{m}\left(\frac{A_n}{2}\right).
\end{align*}

Using~(\ref{U-factorization}) and the fact that the Chebyshev polynomials of the
second kind satisfy
\[
U_j(\cos\theta) = \dfrac{\sin((j+1)\theta)}{\sin\theta},
\]
it is easy to check that the polynomial 
\[
p(x):= U_{m}\left(\frac{x}{2}\right) - U_{m-1}\left(\frac{x}{2}\right)
\]
is a monic polynomial of degree $m$ with zeros
\[
t_{h,m}:=2\cos\dfrac{(2h+1)\pi}{2m+1},\quad0\leq h\leq m-1.
\]
Thus,
\[
T= -p(A_n) = 
-\prod_{h=0}^{m-1}\left({A_n-t_{h,m} I_n}\right),
\]
and by equation~\eqref{eqn:thm1-det}, 
\[
  \det\tD^G=\prod_{h=0}^{m-1}\chi_n(t_{h,m}),
\]
where $\chi_n(x)$ is the characteristic polynomial of $A_n$.  Therefore, to show
that the expressions in parts~(\ref{thm1-1}) and~(\ref{thm1-3}) are equal, it
suffices to show that
\begin{equation}\label{eqn:claim1}
  \chi_n(t_{h,m})=(-1)^n\,U_{2n}(i\,\xi_{m-h,m})
\end{equation}
for each $h\in \{0,1,\cdots, m-1\}$, which we do by showing that both sides of the equation satisfy the same
recurrence.

Define $\chi_0(x):=1$.  Expanding the
determinant defining~$\chi_n(x)$, starting along the first row, leads to 
a recursive formula for $\chi_n(x)$:
\begin{align}\label{chi-recurrence}
\nonumber
\chi_0(x)&=1\\
\chi_1(x)&=3-x\\
\nonumber
\chi_{j}(x)&=(4-x)\chi_{j-1}(x)-\chi_{j-2}(x)\quad \text{for $j\geq2$}.
\end{align}
On the other hand, defining $C_j(x):=(-1)^j\,U_{2j}(x)$, it follows from~\eqref{eqn:2nd} that 
\begin{align}\label{c-recurrence}
\nonumber
C_0(x)&=1\\
C_1(x)&=1-4x^2\\
\nonumber
C_{j}(x)&=(2-4x^2)C_{j-1}(x)-C_{j-2}(x)\quad \text{for $j\geq2$}.
\end{align}
The result now follows by letting $x=t_{h,m}$ in~(\ref{chi-recurrence}),
letting $x=i\,\xi_{m-h,m}$ in~(\ref{c-recurrence}), and using the fact that
\begin{equation}\label{eqn:t-chi}
  t_{h,m}= 2-4\,\xi_{m-h,m}^2.
\end{equation}
(Equation~\eqref{eqn:t-chi} can be verified using, for example, the double-angle
formula for cosine and the relation among angles, $(2h+1)\pi/(2m+1)=\pi-2(m-h)\pi/(2m+1))$.
\medskip

\noindent [{\bf(\ref{thm1-3}) $=$ (\ref{thm1-4})}]:
Using~(\ref{U-factorization}),
\begin{align*}
  (-1)^{mn}&\prod_{h=1}^{m}U_{2n}(i\,\xi_{h,m})\\
  &=(-1)^{mn}\prod_{h=1}^m\prod_{k=1}^{2n}(2i\,\xi_{h,m}-2\,\xi_{k,n})\\
  &=(-1)^{mn}\prod_{h=1}^m\prod_{k=1}^{n}
  (2i\,\xi_{h,m}-2\,\xi_{k,n})
  (2i\,\xi_{h,m}+2\,\xi_{k,n})\\
  &=\prod_{h=1}^m\prod_{k=1}^{n}(4\,\xi_{h,m}^2+4\,\xi_{k,n}^2).
\end{align*}
\end{proof}

\begin{example}\label{example:symm4x4} Figure~\ref{fig:symm4x4} lists the $36$
  symmetric recurrents on $\sg_{4\times 4}$ in no particular order.  Given a
  symmetric recurrent, $c$, let $\tilde{c}$ be the restriction of $c$ to the
  vertices $(1,1)$, $(1,2)$, $(2,1)$, and $(2,2)$, representing the orbits of
  the Klein $4$-group action on~$\sg_{4\times 4}$.  We regard $\tilde{c}$ as a
  configuration on $D_{2\times2}$, the sandpile graph introduced in the proof of
  Theorem~\ref{thm1}.  Let $\iota(c)$ be the recurrent element of the sandpile
  graph $D_{2\times 2}$ equivalent to $\tilde{c}$ modulo the reduced Laplacian
  of $D_{2\times 2}$.  Then $c\mapsto\iota(c)$ determines a bijection between
  the symmetric recurrents of $\sg_{4\times 4}$ and the recurrents of
  $D_{2\times 2}$.  In~\cite{Holroyd}, it is shown that the sandpile group of a
  graph acts freely and transitively on the set of spanning trees of the graph
  rooted at the sink, i.e., this set of spanning trees is a {\em torsor} for the
  sandpile group.  Thus, via the Temperley bijection, the domino tilings of the
  $4\times 4$ checkerboard, forms a torsor for the group of symmetric recurrents
  on $\sg_{4\times 4}$.  
\end{example}

\begin{figure}[ht] 
\begin{tikzpicture}[scale=0.5]
\def\i{12};
\def\j{0};
\node at (\i+0.5,\j+0.95){\# grains};
\node at (0,0){\includegraphics[height=2.5in]{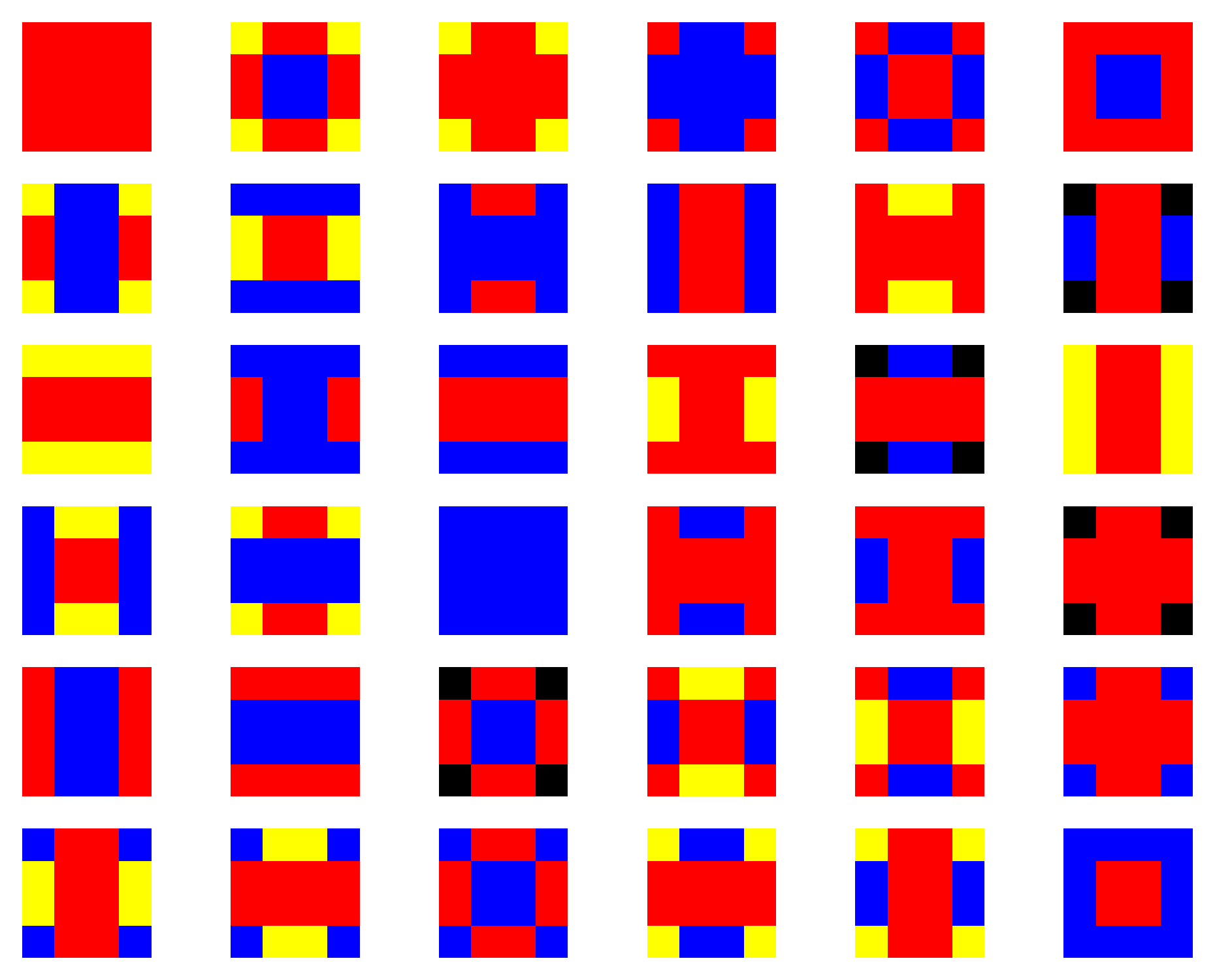}};
\draw[fill=black] (\i,\j) rectangle (\i+0.5,\j-0.5);
\node at (\i+1.4,\j-0.25){$=0$};
\draw[fill=yellow] (\i,\j-1) rectangle (\i+0.5,\j-1.5);
\node at (\i+1.4,\j-1.25){$=1$};
\draw[fill=blue] (\i,\j-2) rectangle (\i+0.5,\j-2.5);
\node at (\i+1.4,\j-2.25){$=2$};
\draw[fill=red] (\i,\j-3) rectangle (\i+0.5,\j-3.5);
\node at (\i+1.4,\j-3.25){$=3$};
\end{tikzpicture}
\caption{The $36$ symmetric recurrents on $\sg_{4\times4}$.}\label{fig:symm4x4}
\end{figure}
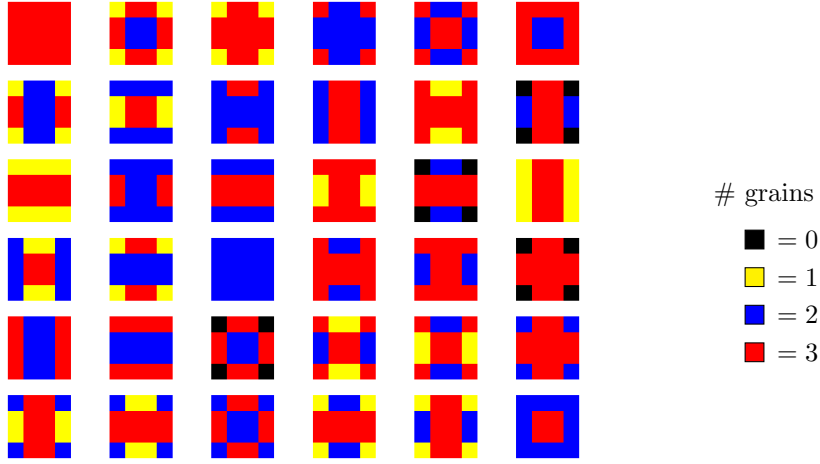

\begin{example}\label{example:main1}  This example illustrates part of the proof
  of Theorem~\ref{thm1} for the case $m=4$ and $n=3$.  Figure~\ref{fig:8x6}
  shows the graph $\sg_{8\times 6}$.  The boxed $4\times 3$ block of vertices in
  the upper left are representatives of the orbits of the Klein $4$-group
  action.  Order these from left-to-right, top-to-bottom, to get the matrix for
  the symmetrized reduced Laplacian, $\tD^G_{8\times 6}$.  The vertex $(2,3)$ of
  $\sg_{8\time6}$ in Figure~\ref{fig:8x6} is colored blue.  If this vertex is
  fired simultaneously with the other vertices in its orbit, it will lose $4$
  grains of sand to its neighbors but gain $1$ grain of sand from the adjacent
  vertex in its orbit.  This firing-rule is encoded in the sixth column of
  $\tD^G_{8\times6}$ (shaded blue).

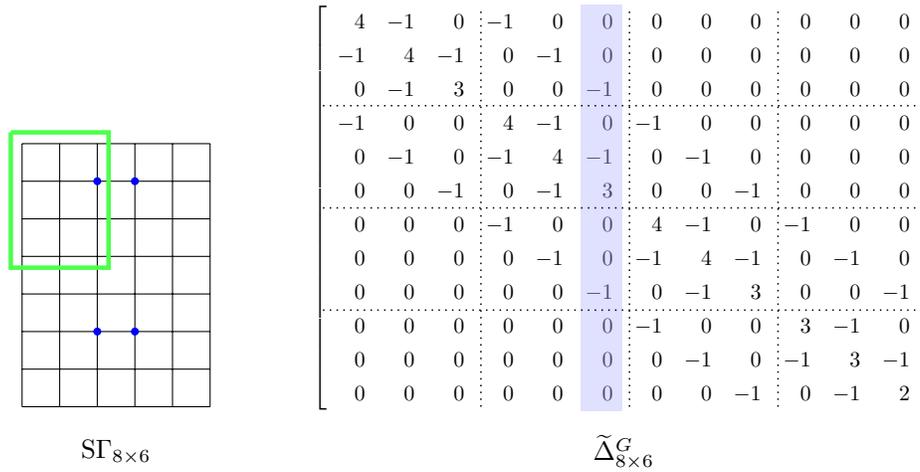
\begin{figure}[ht] 
\begin{tikzpicture}[scale=0.5]

\begin{scope}[shift={(0,-1.8)}]
\draw (0,0) grid (5,7);

\draw[line width=0.6mm,color=green!70] (-0.3,7.3) -- (2.3,7.3) -- (2.3,3.7) -- (-0.3,3.7) -- (-0.3,7.3);

\fill[color=blue] (2,6) circle (3pt);
\fill[color=blue] (3,6) circle (3pt);
\fill[color=blue] (2,2) circle (3pt);
\fill[color=blue] (3,2) circle (3pt);
\end{scope}

\draw (16,3.5) node(symlap){
\resizebox{0.67\columnwidth}{!}{
$
\renewcommand{\arraystretch}{1.3}
\left[
\begin{array}{rrrrrrrrrrrr}
  4&-1& 0&-1& 0& 0& 0& 0& 0& 0& 0& 0\\
 -1& 4&-1& 0&-1& 0& 0& 0& 0& 0& 0& 0\\
  0&-1& 3& 0& 0&-1& 0& 0& 0& 0& 0& 0\\
 -1& 0& 0& 4&-1& 0&-1& 0& 0& 0& 0& 0\\
  0&-1& 0&-1& 4&-1& 0&-1& 0& 0& 0& 0\\
  0& 0&-1& 0&-1& 3& 0& 0&-1& 0& 0& 0\\
  0& 0& 0&-1& 0& 0& 4&-1& 0&-1& 0& 0\\
  0& 0& 0& 0&-1& 0&-1& 4&-1& 0&-1& 0\\
  0& 0& 0& 0& 0&-1& 0&-1& 3& 0& 0&-1\\
  0& 0& 0& 0& 0& 0&-1& 0& 0& 3&-1& 0\\
  0& 0& 0& 0& 0& 0& 0&-1& 0&-1& 3&-1\\
  0& 0& 0& 0& 0& 0& 0& 0&-1& 0&-1& 2\\
\end{array}
\right]
$
}
};

\fill[color=blue!30,opacity=0.4] (14.85,-1.9) rectangle (15.95,8.9);

\foreach \i in {1,2,3} {
  \draw[line width=0.02cm, dotted] (8.0,2.72*\i-1.95) --(24,2.72*\i-1.95);
}
\draw[line width=0.02cm, dotted] (12.2,-1.8) --(12.2,8.8);
\draw[line width=0.02cm, dotted] (16.15,-1.8) --(16.15,8.8);
\draw[line width=0.02cm, dotted] (20.1,-1.8) --(20.1,8.8);

\draw (2.5,-3) node(a) {$\sg_{8\times 6}$};
\draw (16,-3) node(b) {$\tD^{G}_{8\times6}$};

\end{tikzpicture}
\caption{A sandpile grid graph and its symmetrized reduced Laplacian.}\label{fig:8x6}
\end{figure}

The matrix $\tD^G_{8\times 6}$ is the reduced Laplacian of the graph $D_{4\times
3}$, shown in Figure~\ref{fig:4x3}.  To form $\mathcal{H}(D_{4\times
3})=\Gamma_{8\times 6}$, we first overlay $D_{4\times 3}$ with its dual, as
shown, then remove the vertices $s$ and $\tilde{s}$ and their incident edges.
Figure~\ref{fig:entwined} shows how a spanning tree of $D_{4\times 3}$ (in black)
determines a spanning tree of the dual graph (in blue) and a domino tiling of
the $8\times 6$ checkerboard.

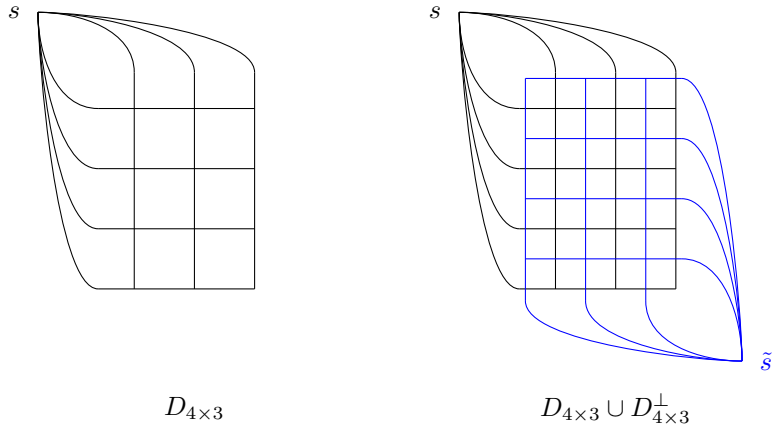
\begin{figure}[ht] 
\begin{tikzpicture}[scale=0.8]

\foreach \i in {0,1,2} {
\draw (\i,0) -- (\i,3.6);
\draw (\i,3.6) .. controls +(90:20pt) and +(0:20pt) .. (-1.60,4.6);
}
\foreach \i in {0,1,2,3} {
\draw (-0.60,\i) -- (2,\i);
\draw (-0.60,\i) .. controls +(180:20pt) and +(270:20pt) .. (-1.60,4.6);
}

\draw (-2,4.6) node(s){$s$};
\draw (1,-2) node(d){$D_{4\times3}$};

\def\j{7};
\foreach \i in {0,1,2} {
\draw (\i+\j,0) -- (\i+\j,3.6);
\draw (\i+\j,3.6) .. controls +(90:20pt) and +(0:20pt) .. (-1.60+\j,4.6);
}
\foreach \i in {0,1,2,3} {
\draw (-0.60+\j,\i) -- (2+\j,\i);
\draw (-0.60+\j,\i) .. controls +(180:20pt) and +(270:20pt) .. (-1.60+\j,4.6);
}
\draw (-2+\j,4.6) node(s){$s$};
\foreach \i in {0,1,2} {
\draw[color=blue] (\i+\j-0.5,-0.2) -- (\i+\j-0.5,3.5);
\draw[color=blue] (\i+\j-0.5,-0.2) .. controls +(270:20pt) and +(180:20pt) .. (3.1+\j,-1.2);
}
\foreach \i in {0,1,2,3} {
\draw[color=blue] (2.10+\j,\i+0.5) -- (-0.50+\j,\i+0.5);
\draw[color=blue] (2.10+\j,\i+0.5) .. controls +(0:20pt) and +(90:20pt) .. (3.1+\j,-1.2);
}
\draw (3.1+\j+0.4,-1.2) node(s){$\color{blue}\tilde{s}$};
\draw (1+\j,-2) node(e){$D_{4\times3}\cup D_{4\times3}^{\perp}$};
\end{tikzpicture}
\caption{The symmetrized reduced Laplacian for $\sg_{8\times 6}$ is the reduced
Laplacian for $D_{4\times 3}$. Removing $s$ and $\tilde{s}$ and their incident
edges from the graph on the right shows $\mathcal{H}(D_{4\times
3})=\Gamma_{8\times 6}$.}\label{fig:4x3}
\end{figure}
\begin{figure}[ht] 
\begin{tikzpicture}[scale=0.8]

\foreach \i in {0,2,4} {
  \foreach \j in {0,2,4,6} {
  \draw[fill=white] (\i,\j) rectangle (\i+1,\j+1);
  \draw[fill=black!65] (\i+1,\j) rectangle (\i+2,\j+1);
  \draw[fill=blue!65] (\i,\j+1) rectangle (\i+1,\j+2);
  \draw[fill=white] (\i+1,\j+1) rectangle (\i+2,\j+2);
  }
}

\filldraw[draw=black,style=ultra thick,fill=blue,rounded corners] (0.2,0.2) rectangle (0.8,1.8);
\filldraw[draw=black,style=ultra thick,fill=blue,rounded corners] (0.2,2.2) rectangle (0.8,3.8);
\filldraw[draw=black,style=ultra thick,fill=blue,rounded corners] (0.2,5.2) rectangle (1.8,5.8);
\filldraw[draw=black,style=ultra thick,fill=blue,rounded corners] (0.2,6.2) rectangle (0.8,7.8);
\filldraw[draw=white,fill=black,rounded corners] (1.2,6.2) rectangle (1.8,7.8);
\filldraw[draw=white,fill=black,rounded corners] (0.2,4.2) rectangle (1.8,4.8);
\filldraw[draw=white,fill=black,rounded corners] (1.2,2.2) rectangle (1.8,3.8);
\filldraw[draw=black,style=ultra thick,fill=blue,rounded corners] (1.2,1.2) rectangle (2.8,1.8);
\filldraw[draw=white,fill=black,rounded corners] (1.2,0.2) rectangle (2.8,0.8);
\filldraw[draw=black,style=ultra thick,fill=blue,rounded corners] (2.2,7.2) rectangle (3.8,7.8);
\filldraw[draw=white,fill=black,rounded corners] (2.2,6.2) rectangle (3.8,6.8);
\filldraw[draw=white,fill=black,rounded corners] (3.2,4.2) rectangle (3.8,5.8);
\filldraw[draw=black,style=ultra thick,fill=blue,rounded corners] (2.2,4.2) rectangle (2.8,5.8);
\filldraw[draw=black,style=ultra thick,fill=blue,rounded corners] (2.2,3.2) rectangle (3.8,3.8);
\filldraw[draw=white,fill=black,rounded corners] (2.2,2.2) rectangle (3.8,2.8);
\filldraw[draw=white,fill=black,rounded corners] (3.2,0.2) rectangle (3.8,1.8);
\filldraw[draw=white,fill=black,rounded corners] (4.2,0.2) rectangle (5.8,0.8);
\filldraw[draw=black,style=ultra thick,fill=blue,rounded corners] (4.2,7.2) rectangle (5.8,7.8);
\filldraw[draw=black,style=ultra thick,fill=blue,rounded corners] (4.2,5.2) rectangle (4.8,6.8);
\filldraw[draw=white,fill=black,rounded corners] (5.2,5.2) rectangle (5.8,6.8);
\filldraw[draw=white,fill=black,rounded corners] (4.2,4.2) rectangle (5.8,4.8);
\filldraw[draw=black,style=ultra thick,fill=blue,rounded corners] (4.2,2.2) rectangle (4.8,3.8);
\filldraw[draw=white,fill=black,rounded corners] (5.2,2.2) rectangle (5.8,3.8);
\filldraw[draw=black,style=ultra thick,fill=blue,rounded corners] (4.2,1.2) rectangle (5.8,1.8);

\draw[fill=black] (-0.8,8.8) circle (1.5pt);

\draw[style=ultra thick] (-0.8,8.8) .. controls (1.5,8.4) .. (1.5,8);
\draw[style=thin] (-0.8,8.8) .. controls (3.5,8.4) .. (3.5,8);
\draw[style=thin] (-0.8,8.8) .. controls (5.5,8.4) .. (5.5,8);

\draw[style=thin] (-0.8,8.8) .. controls (-0.4,6.5) .. (0,6.5);
\draw[style=ultra thick] (-0.8,8.8) .. controls (-0.4,4.5) .. (0,4.5);
\draw[style=thin] (-0.8,8.8) .. controls (-0.4,2.5) .. (0,2.5);
\draw[style=thin] (-0.8,8.8) .. controls (-0.4,0.5) .. (0,0.5);

\draw[style=ultra thick] (1.5,8) -- (1.5,7.8);
\draw[style=ultra thick] (1.8,6.5) -- (2.2,6.5);
\draw[style=ultra thick] (3.5,6.2) -- (3.5,5.8);
\draw[style=ultra thick] (3.8,4.5) -- (4.2,4.5);
\draw[style=ultra thick] (5.5,4.8) -- (5.5,5.2);
\draw[style=ultra thick] (5.5,4.2) -- (5.5,3.8);
\draw[style=ultra thick] (0.0,4.5) -- (0.2,4.5);
\draw[style=ultra thick] (1.5,4.2) -- (1.5,3.8);
\draw[style=ultra thick] (1.8,2.5) -- (2.2,2.5);
\draw[style=ultra thick] (3.5,2.2) -- (3.5,1.8);
\draw[style=ultra thick] (2.8,0.5) -- (3.2,0.5);
\draw[style=ultra thick] (3.8,0.5) -- (4.2,0.5);

\draw[fill=blue] (6.8,-0.8) circle (1.5pt);

\draw[style=ultra thick,color=blue] (6.8,-0.8) .. controls (6.4,1.5) .. (6,1.5);
\draw[style=thin,color=blue] (6.8,-0.8) .. controls (6.4,3.5) .. (6,3.5);
\draw[style=thin,color=blue] (6.8,-0.8) .. controls (6.4,5.5) .. (6,5.5);
\draw[style=ultra thick,color=blue] (6.8,-0.8) .. controls (6.4,7.5) .. (6,7.5);

\draw[style=thin,color=blue] (6.8,-0.8) .. controls (4.5,-0.4) .. (4.5,0);
\draw[style=thin,color=blue] (6.8,-0.8) .. controls (2.5,-0.4) .. (2.5,0);
\draw[style=ultra thick,color=blue] (6.8,-0.8) .. controls (0.5,-0.4) .. (0.5,0);

\draw[style=ultra thick,color=blue] (0.5,6.2) -- (0.5,5.8);
\draw[style=ultra thick,color=blue] (1.8,5.5) -- (2.2,5.5);
\draw[style=ultra thick,color=blue] (2.5,4.2) -- (2.5,3.8);
\draw[style=ultra thick,color=blue] (3.8,3.5) -- (4.2,3.5);
\draw[style=ultra thick,color=blue] (4.5,2.2) -- (4.5,1.8);
\draw[style=ultra thick,color=blue] (5.8,1.5) -- (6.0,1.5);
\draw[style=ultra thick,color=blue] (0.5,2.2) -- (0.5,1.8);
\draw[style=ultra thick,color=blue] (0.8,1.5) -- (1.2,1.5);
\draw[style=ultra thick,color=blue] (0.5,0.2) -- (0.5,0.0);
\draw[style=ultra thick,color=blue] (0.5,0.2) -- (0.5,0.0);
\draw[style=ultra thick,color=blue] (3.8,7.5) -- (4.2,7.5);
\draw[style=ultra thick,color=blue] (5.8,7.5) -- (6.0,7.5);
\draw[style=ultra thick,color=blue] (4.5,7.2) -- (4.5,6.8);
\end{tikzpicture}
\caption{Every domino tiling of an even-sided checkerboard consists of a spanning
tree entwined with its dual spanning tree.}\label{fig:entwined}
\end{figure}
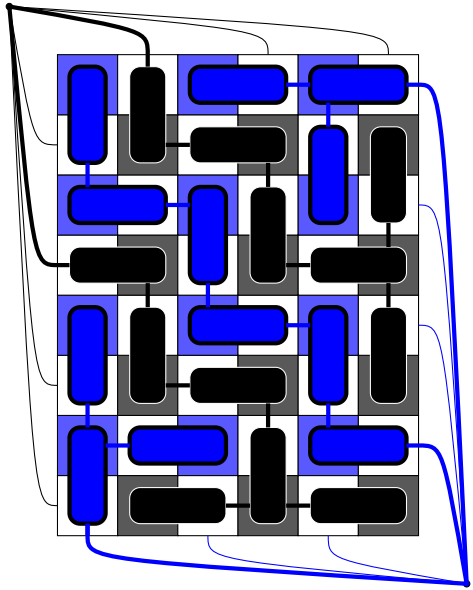
\end{example}

\subsection{Symmetric recurrents on a $2m\times (2n-1)$ sandpile grid
graph.}\label{subsection:symmetric recurrents on evenxodd grid} The {\em
$m\times n$ M\"obius grid graph}, $\mg_{m\times n}$, is the graph formed from
the ordinary $m\times n$ grid graph, $\Gamma_{m\times n}$, by adding the edges
$\{(h,1),(m-h+1,n)\}$ for $1\leq h\leq m$.  A {\em M\"obius checkerboard} is an
ordinary checkerboard with its left and right sides glued with a twist.  Domino
tilings of an $m\times n$ M\"obius checkerboard are identified with perfect
matchings of $\mg_{m\times n}$.  See Figure~\ref{fig:Moebius} for examples. 
\begin{figure}[ht] 
\begin{tikzpicture}[scale=0.7]

\SetVertexMath
\GraphInit[vstyle=Art]
\SetUpVertex[MinSize=3pt]
\SetVertexLabel
\tikzset{VertexStyle/.style = {%
shape = circle,
shading = ball,
ball color = black,
inner sep = 0pt
}}
\SetUpEdge[color=black]
\foreach \i in {0,...,3} {
  \draw (\i,0) -- (\i,3);
  \draw (-0.5,\i) -- (3.5,\i);
  \draw[dotted, thick] (3.5,\i) -- (4.0,\i);
  \draw[dotted, thick] (-1.0,\i) -- (-0.5,\i);
}
\foreach \i in {0,1,2,3}{
  \foreach \j in {0,1,2,3}{
    \Vertex[NoLabel,x=\i,y=\j]{}
  }
}

\node [right] at (4,3){d};
\node [left] at (-1.0,0){d};
\node [right] at (4,2){c};
\node [left] at (-1.0,1){c};
\node [right] at (4,1){b};
\node [left] at (-1.0,2){b};
\node [right] at (4,0){a};
\node [left] at (-1.0,3){a};

\node at (1.5,-1) {(i)};

\def\xoff{7}
\def\yoff{-0.5}
\foreach \i in {0,...,4} {
  \draw (\xoff+\i,\yoff+0) -- (\xoff+\i,\yoff+4);
  \draw (\xoff+0,\yoff+\i) -- (\xoff+4,\yoff+\i);
}
\def\l{0.9}
\def\s{0.1}
\draw[fill=black!10,rounded corners] (\xoff-\l,\yoff+3+\s) rectangle (\xoff+\l,\yoff+3+\l);
\draw[fill=black!10,rounded corners] (\xoff-\l,\yoff+2+\s) rectangle (\xoff+\l,\yoff+2+\l);
\draw[fill=black!10,rounded corners] (\xoff+1+\s,\yoff+2+\s) rectangle (\xoff+1+\l,\yoff+3+\l);
\draw[fill=black!10,rounded corners] (\xoff+2+\s,\yoff+2+\s) rectangle (\xoff+3+\l,\yoff+2+\l);
\draw[fill=black!10,rounded corners] (\xoff+2+\s,\yoff+3+\s) rectangle (\xoff+3+\l,\yoff+3+\l);
\draw[fill=black!10,rounded corners] (\xoff+\s,\yoff+\s) rectangle (\xoff+\l,\yoff+1+\l);
\draw[fill=black!10,rounded corners] (\xoff+1+\s,\yoff+\s) rectangle (\xoff+2+\l,\yoff+\l);
\draw[fill=black!10,rounded corners] (\xoff+1+\s,\yoff+1+\s) rectangle (\xoff+2+\l,\yoff+1+\l);
\draw[fill=black!10,rounded corners] (\xoff+3+\s,\yoff+\s) rectangle (\xoff+4+\l,\yoff+\l);
\draw[fill=black!10,rounded corners] (\xoff+3+\s,\yoff+1+\s) rectangle (\xoff+4+\l,\yoff+1+\l);
\node at (\xoff+0,\yoff+3.5) {A};
\node at (\xoff+0,\yoff+2.5) {B};
\node at (\xoff+4,\yoff+0.5) {A};
\node at (\xoff+4,\yoff+1.5) {B};

\node at (\xoff+2,\yoff-0.5) {(ii)};
\end{tikzpicture}
\caption{(i) The $4\times4$ M\"obius grid graph,
$\mg_{4\times4}$; (ii) A tiling of the $4\times4$ M\"obius
checkerboard.}\label{fig:Moebius}
\end{figure}
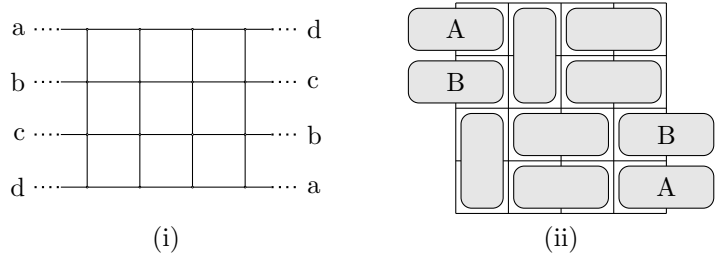

As part of Theorem~\ref{thm2}, we will show that the domino tilings of a
$2m\times 2n$ M\"obius checkerboard can be counted using weighted domino tilings
of an associated ordinary checkerboard, which we now describe. Define the {\em
M\"obius-weighted $m\times n$ grid graph}, $\tg_{m\times n}$, as the ordinary
$m\times n$ grid graph but with each edge of the form $\{(m-2h,n-1),(m-2h,n)\}$
for $0\leq h<\lfloor\frac{m}{2}\rfloor$ assigned the weight~$2$, and, if~$m$
is odd, then in addition assign the edge $\{(1,n-1),(1,n)\}$ the weight~$3$ (and all
other edges have weight~$1$). (In the case $m=1$, the weight of the edge
$\{(1,n-1),(1,n)\}$ is be defined to be $3$.)   
See Figure~\ref{fig:mobius-weighted} for examples. 
\begin{figure}[ht] 
\begin{tikzpicture}[scale=0.5]
\draw (0,0) grid(3,3);
\foreach \j in {0,2}{
  \draw[fill=white, draw=none] (2.5,\j) circle [radius=2.1mm];
  \draw (2.5,\j) node{$2$};
}
\node at (1.8,-1){$\tg_{4\times 4}$}; 

\def\i{8}
\draw (\i+0,0) grid(\i+1,4);
\foreach \j in {0,2}{
  \draw[fill=white, draw=none] (\i+0.5,\j) circle [radius=2.8mm];
  \draw (\i+0.5,\j) node{$2$};
}
\draw[fill=white, draw=none] (\i+0.5,4) circle [radius=2.8mm];
\draw (\i+0.5,4) node{$3$};
\node at (\i+0.8,-1){$\tg_{5\times 2}$}; 
\end{tikzpicture}
\caption{Two M\"obius-weighted grid graphs.}\label{fig:mobius-weighted}
\end{figure}
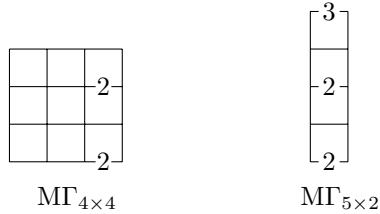
The {\em M\"obius-weighted $m\times n$ checkerboard} is the ordinary $m\times n$
checkerboard but for which the weight of a domino tiling is taken to be the
weight of the corresponding perfect matching of $\tg_{m\times n}$.  In
Figure~\ref{fig:checker4x4}, the dominos corresponding to edges of weight~$2$
are shaded.  Thus, the first three tilings in the first row of
Figure~\ref{fig:checker4x4} have weights $4$, $2$, and $1$, respectively.
Example~\ref{example:mobius 3x1} considers a case for which $m$ is odd.

\begin{thm}\label{thm2} Let $T_j(x)$ denote the $j$-th Chebyshev polynomial of
  the first kind, and let
  \[
  \xi_{h,d}:=\cos\left(\frac{h\pi}{2d+1}\right)\quad\text{and}\quad
  \zeta_{h,d}:=\cos\left(\frac{(2h-1)\pi}{4d}\right)
  \]
  for all integers $h$ and $d\neq0$.
  Then for all integers $m,n\geq 1$, the following are equal:
\begin{enumerate}
  \item\label{thm2-1} the number of symmetric recurrents on $\sg_{2m\times(2n-1)}$;
  \item\label{thm2-2} if $n>1$, the number of domino tilings of the
    M\"obius-weighted $2m\times 2n$ checkerboard, and if $n=1$, the number of
    domino tilings of the M\"obius-weighted $(2m-1)\times 2$ checkerboard, 
    counted according to weight;
  \item\label{thm2-3} \ 
    \[
    (-1)^{mn}\,2^m\prod_{h=1}^m T_{2n}(i\,\xi_{h,m});
    \]
  \item\label{thm2-4} \
    \[
    \prod_{h=1}^{m}\prod_{k=1}^{n}\left(4\,\xi_{h,m}^2+4\,\zeta_{k,n}^2\right);
    \]
  \item\label{thm2-5} the number of domino tilings of a $2m\times 2n$
    M\"obius checkerboard.
\end{enumerate}
\end{thm}
\begin{remark}\label{remark:chebT}
  By identity~(\ref{eqn:half-angle}),
  \[
  T_{2n}(i\,\xi_{h,m})=(-1)^n\,T_n(1+2\,\xi_{h,m}^2),
  \]
  from which it follows, after proving Theorem~\ref{thm2}, that
  \[
  2^m\prod_{h=1}^mT_n(1+2\,\xi_{h,m}^2)
  \]
  is another way to express the numbers in parts (1)--(5). 
\end{remark}

\begin{proof}[Proof of Theorem~\ref{thm2}.]  The proof is similar to that of
  Theorem~\ref{thm1} after altering the definitions of the matrices $A_n$ and
  $B_n$ used there.  This time, for $n>1$, let $A'_n=(a'_{h,k})$ be the $n\times n$
  tridiagonal matrix with entries
\[
a'_{h,k} = 
\begin{cases}
  \hfill 4 & \quad \text{if $h=k$},\\
  \hfill-1 & \quad \text{if $|h-k|=1$ and $h\neq n$},\\
  -2& \quad \text{if $h=n$ and $k=n-1$},\\
  \hfill 0 & \quad \text{if $|h-k|\geq 2$}. 
\end{cases}
\]
In particular, $A'_1=[4]$.  Define the matrix $B'_n=(b'_{h,k})$ by
\[
b'_{h,k} = 
\begin{cases}
  \ 3 & \quad \text{if $h=k$},\\
  \ a'_{h,k} & \quad \text{otherwise}.
\end{cases}
\]
Thus, for instance,
\[
A'_3=
\left[\begin{array}{rrr}
4 & -1 & 0 \\
-1 & 4 & -1 \\
0 & -2 & 4
\end{array}\right],\qquad
B'_3=
\left[\begin{array}{rrr}
3 & -1 & 0 \\
-1 & 3 & -1 \\
0 & -2 & 3
\end{array}\right].
\]
If $n=1$, take $A'_1=[4]$ and $B'_1=[3]$.
\medskip

\noindent [{\bf(\ref{thm2-1}) $=$ (\ref{thm2-2})}]: Reasoning as in the proof of
Theorem~\ref{thm1}, equation~\eqref{eqn:deltaG} with $A'_n$ and~$B'_n$
substituted for $A_n$ and $B_n$ gives the symmetrized reduced Laplacian,
$\tD^G$, of $\sg_{2m\times(2n-1)}$.  Unless $n=1$, the matrix $\tD^G$ is {\em
not} the reduced Laplacian matrix of a sandpile graph since the sum of the
elements in its penultimate column is~$-1$ whereas the sum of the elements in
any column of the reduced Laplacian of a sandpile graph must be nonnegative.
However, in any case, the transpose $(\tD^G)^t$ is the reduced Laplacian of a
sandpile graph, which we call~$D'_{m\times n}$.  We embed it in the plane as a
grid as we did previously with $D_{m\times n}$ in the proof of
Theorem~\ref{thm1}, but this time with some edge-weights not equal to $1$.

Figure~\ref{fig:mobius 4x3} shows~$D'_{4\times3}$.  It is the same as
$D_{4\times3}$ as depicted in Figure~\ref{fig:4x3}, except that arrowed edges,
\raisebox{0.5ex}{\tikz\draw[<<->,>=angle 60] (0,0)--(0.8,0);}, have been
substituted for certain edges.  Each represents a pair of arrows---one from
right-to-left of weight~$2$ and one from left-to-right of weight~$1$---embedded
so that they coincide, as discussed in Section~\ref{section:Matchings and
trees}.  
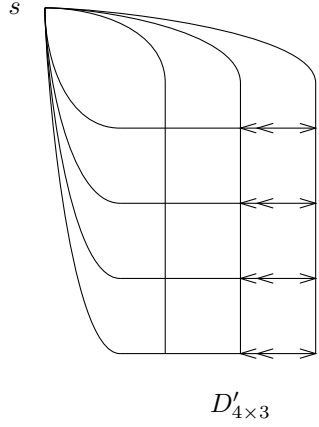
\begin{figure}[ht] 
\begin{tikzpicture}[scale=1]

\foreach \i in {0,1,2} {
\draw (\i,0) -- (\i,3.6);
\draw (\i,3.6) .. controls +(90:20pt) and +(0:20pt) .. (-1.60,4.6);
}
\foreach \i in {0,1,2,3} {
\draw (-0.60,\i) -- (1,\i);
\draw[<<->,>=mytip] (1,\i) -- (2,\i);
\draw (-0.60,\i) .. controls +(180:20pt) and +(270:20pt) .. (-1.60,4.6);
}

\draw (-2,4.6) node(s){$s$};
\draw (1,-0.7) node(d){$D'_{4\times3}$};
\end{tikzpicture}
\caption{The symmetrized reduced Laplacian for $\sg_{8\times 5}$ is the reduced
Laplacian for $D'_{4\times 3}$.  Arrowed edges each represent a pair of directed
edges of weights~$1$ and $2$, respectively, as indicated by the number of arrow
heads. All other edges have weight $1$.}\label{fig:mobius 4x3}
\end{figure}

Reasoning as in the proof of Theorem~\ref{thm1}, we see that the number of
perfect matchings of $\mathcal{H}(D'_{m\times n})$ is equal to the number of
perfect matchings of $\tg_{2m\times(2n-1)}$, each counted according to weight.
This number is $\det(\tD^G)^t=\det\tD^G$, which is the number of symmetric
recurrents on $\sg_{2m\times(2n-1)}$ by Corollary~\ref{cor:nsr}.
\medskip

\noindent [{\bf(\ref{thm2-1}) $=$ (\ref{thm2-3})}]: Exactly the same argument as
given in the proof of Theorem~\ref{thm1} shows that
\[
\det\tD^G=\prod_{h=0}^{m-1}\chi_n(t_{h,m}),
\]
where $t_{h,m}$ is as before, but now $\chi_n(x)$ is the characteristic
polymonial of $A'_n$.  In light of Remark~\ref{remark:chebT}, it suffices to
show
\[
\chi_n(t_{h,m})=2\,T_n(1+2\,\xi_{m-h,m}^2)
\]
for each $h\in \{0,1,\cdots, m-1\}$ , which we now do as before, by showing both sides of the equation satisfy the same
recurrence.

Defining $\chi_0(x):=2$ and expanding the determinant defining $\chi_n(x)$ yields
\begin{align}
\chi_0(x)&=2\notag\\
\chi_1(x)&=4-x\label{thm2-chi}\\
\chi_j(x)&=(4-x)\chi_{j-1}(x)-\chi_{j-2}(x)\quad \text{for $j\geq2$}.\notag
\end{align}
On the other hand, definining $C_j(x):=2\,T_{j}(x)$, it follows from~\eqref{eqn:1st} that 
\begin{align}
C_0(x)&=2\notag\\
C_1(x)&=2x\label{thm2-C}\\
C_{j}(x)&=2x\,C_{j-1}(x)-C_{j-2}(x)\quad \text{for $j\geq2$}.\notag
\end{align}
The result now follows as before, using equation~\eqref{eqn:t-chi}.
\medskip

\noindent [{\bf(\ref{thm2-3}) $=$ (\ref{thm2-4})}]:  The numbers given in
parts~(\ref{thm2-3}) and~(\ref{thm2-4}) are equal by a straightforward
calculation, similar to that in the proof of the analogous result in
Theorem~\ref{thm1}, this time using~(\ref{T-factorization}).
\medskip

\noindent [{\bf(\ref{thm2-4}) $=$ (\ref{thm2-5})}]:  Formula~(2)
in~\cite{LW} gives the number of domino tilings of a $2m\times2n$ M\"obius
checkerboard.  That formula is identical to our double-product in
part~(\ref{thm2-4}) but with $\sin((4k-1)\pi/(4n))$ substituted for
$\zeta_{k,n}$.  Now,
\[
\sin\left( \frac{(4k-1)\pi}{4n}\right)=\cos\left(\frac{(4k-1-2n)\pi}{4n}\right).
\]
Defining $\theta(k)=(2k-1)\pi/(4n)$ and $\psi(k)=(4k-1-2n)\pi/(4n)$, it
therefore suffices to show that there is a permutation $\sigma$ of
$\{1,\dots,n\}$ such that $\theta(k)=\pm\psi(\sigma(k))$ for $k=1,\dots,n$
because, in that case, $\zeta_{k,n}=\cos(\theta(k))=\cos(\psi(\sigma(k)))$.
Such a permutation exists, for if $n=2t$, then 
\begin{align*}
\theta(2\ell-1)&=-\psi(t-\ell+1),\qquad1\leq \ell\leq t,\\
\theta(2\ell)&=\psi(\ell+t),\qquad1\leq \ell\leq t,
\end{align*}
and if $n=2t-1$, then
\begin{align*}
\theta(2\ell-1)&=\psi(t+\ell-1),\qquad1\leq \ell\leq t,\\
\theta(2\ell)&=-\psi(t-\ell),\qquad1\leq \ell\leq t-1.
\end{align*}

\end{proof}
\begin{remark}\label{remark:Lu-Wu} In the proof of Theorem~\ref{thm2}, we
  rewrote the double-product in part~(\ref{thm2-4}) as the Lu-Wu formula
  ((2)~in~\cite{LW}) for the number of domino tilings of the $2m\times2n$
  M\"obius checkerboard:
  \[
    \prod_{h=1}^{m}\prod_{k=1}^{n}\left(4\,\xi_{h,m}^2+4\,\mu_{k,n}^2\right),
  \]
  where
   $\mu_{k,n}:=\sin( (4k-1)\pi/(4n))$.  Thus, it is the work of Lu and Wu
  that allowed us to add part~(\ref{thm2-5}) to Theorem~\ref{thm2}.  This is in
  contrast to Theorem~\ref{thm1}, which gave an independent proof of the
  Kastelyn and Temperley-Fisher formula for the number of tilings of the
  ordinary $2m\times2n$ checkerboard.  
\end{remark}
\begin{example}
  The $36$ tilings of the ordinary $4\times4$ checkerboard are listed in
  Figure~\ref{fig:checker4x4}.  Considering these as tilings of the
  M\"obius-weighted $4\times 4$ checkerboard, the sum of the weights of the
  tilings is $71$, which is the number of tilings of the $4\times 4$ M\"obius
  checkerboard and the number of symmetric recurrents on $\sg_{4\times 3}$, in
  accordance with Theorem~\ref{thm2}.
\end{example}
\begin{example}\label{example:mobius 3x1}
  Figure~\ref{fig:m=3 n=1} shows the domino tilings of the M\"obius-weighted
  $5\times 2$ checkerboard. The total number of tilings, counted according to
  weight, is $41$, which is the number of domino tilings of a $6\times 2$
  M\"obius checkerboard, in agreement with case $m=3$ and
  $n=1$ of Theorem~\ref{thm2}.
\end{example}
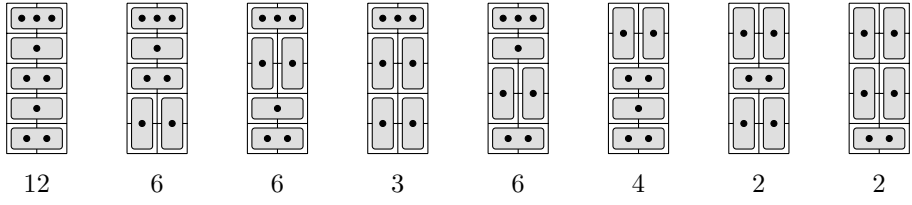
\begin{figure}[ht]
\begin{tikzpicture}[scale=0.4]
\newcommand{\horizRect}[3]{
\draw[fill=gray!25,rounded corners=0.4mm] (#1+0.15,#2+0.15) rectangle (#1+1.85,#2+0.85);
\ifnum#3=1
  \draw[fill=black] (#1+1.0,#2+0.5) circle[radius=0.1];
  \else
    \ifnum#3=2
      \draw[fill=black] (#1+0.666,#2+0.5) circle[radius=0.1];
      \draw[fill=black] (#1+1.333,#2+0.5) circle[radius=0.1];
    \else
      \ifnum#3=3
	\draw[fill=black] (#1+0.5,#2+0.5) circle[radius=0.1];
	\draw[fill=black] (#1+1.0,#2+0.5) circle[radius=0.1];
	\draw[fill=black] (#1+1.5,#2+0.5) circle[radius=0.1];
      \fi  
    \fi
\fi  
}
\newcommand{\vertRect}[3]{
\draw[fill=gray!25,rounded corners=0.4mm] (#1+0.15,#2+0.15) rectangle (#1+0.85,#2+1.85);
\ifnum#3=1
  \draw[fill=black] (#1+0.5,#2+1.0) circle[radius=0.1];
  \else
    \ifnum#3=2
      \draw[fill=black] (#1+0.5,#2+0.666) circle[radius=0.1];
      \draw[fill=black] (#1+0.5,#2+1.333) circle[radius=0.1];
    \else
      \ifnum#3=3
	\draw[fill=black] (#1+0.5,#2+0.5) circle[radius=0.1];
	\draw[fill=black] (#1+0.5,#2+1.0) circle[radius=0.1];
	\draw[fill=black] (#1+0.5,#2+1.5) circle[radius=0.1];
      \fi  
    \fi
\fi  
}

\foreach \i in {0,1,2,3,4,5,6,7}{
  \foreach \j in {0} { 
  \draw (4*\i,7*\j) grid (4*\i+2,7*\j+5);
  }
}
\def\i{0}
\horizRect{\i}{4}{3};
\horizRect{\i}{3}{1};
\horizRect{\i}{2}{2};
\horizRect{\i}{1}{1};
\horizRect{\i}{0}{2};
\draw (\i+1,-1) node{$12$};

\def\i{4}
\horizRect{\i}{4}{3};
\horizRect{\i}{3}{1};
\horizRect{\i}{2}{2};
\vertRect{\i}{0}{1};
\vertRect{\i+1}{0}{1};
\draw (\i+1,-1) node{$6$};

\def\i{8}
\horizRect{\i}{4}{3};
\vertRect{\i}{2}{1};
\vertRect{\i+1}{2}{1};
\horizRect{\i}{1}{1};
\horizRect{\i}{0}{2};
\draw (\i+1,-1) node{$6$};

\def\i{12}
\horizRect{\i}{4}{3};
\vertRect{\i}{2}{1};
\vertRect{\i+1}{2}{1};
\vertRect{\i}{0}{1};
\vertRect{\i+1}{0}{1};
\draw (\i+1,-1) node{$3$};

\def\i{16}
\horizRect{\i}{4}{3};
\horizRect{\i}{3}{1};
\vertRect{\i}{1}{1};
\vertRect{\i+1}{1}{1};
\horizRect{\i}{0}{2};
\draw (\i+1,-1) node{$6$};

\def\i{20}
\vertRect{\i}{3}{1};
\vertRect{\i+1}{3}{1};
\horizRect{\i}{2}{2};
\horizRect{\i}{1}{1};
\horizRect{\i}{0}{2};
\draw (\i+1,-1) node{$4$};

\def\i{24}
\vertRect{\i}{3}{1};
\vertRect{\i+1}{3}{1};
\horizRect{\i}{2}{2};
\vertRect{\i}{0}{1};
\vertRect{\i+1}{0}{1};
\draw (\i+1,-1) node{$2$};

\def\i{28}
\vertRect{\i}{3}{1};
\vertRect{\i+1}{3}{1};
\vertRect{\i}{1}{1};
\vertRect{\i+1}{1}{1};
\horizRect{\i}{0}{2};
\draw (\i+1,-1) node{$2$};
\end{tikzpicture}
\caption{Domino tilings of the M\"obius-weighted $5\times2$ checkerboard.  The
number of dots on each domino indicates its weight.  The weight of each tiling
appears underneath.}\label{fig:m=3 n=1}
\end{figure}
\subsection{Symmetric recurrents on a $(2m-1)\times (2n-1)$ sandpile grid
graph.}\label{subsection:symmetric recurrents on oddxodd grid} The {\em
$2$-weighted $2m\times 2n$ grid graph}, $2$-$\Gamma_{2m\times2n}$ is the
ordinary $2m\times 2n$ grid graph but where each horizontal edge of the form
$\{(2m-2h,2n-1),(2m-2h,2n)\}$ for $0\leq h< m$ and each vertical edge of the
form $\{(2m-1,2n-2k),(2m,2n-2k)\}$ for $0\leq k< n$ is assigned the weight~$2$ (and all other
edges have weight~$1$).  See Figure~~\ref{fig:mobius-weighted} for an example.
The {\em $2$-weighted $2m\times 2n$ checkerboard} is the ordinary $2m\times 2n$
checkerboard but for which the weight of a domino tiling is taken to be the
weight of the corresponding perfect matching of $2$-$\Gamma_{2m\times 2n}$.
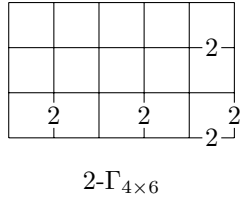
\begin{figure}[ht] 
\begin{tikzpicture}[scale=0.6]
\draw (0,0) grid(5,3);
\foreach \j in {0,2}{
  \draw[fill=white, draw=none] (4.5,\j) circle [radius=2.1mm];
  \draw (4.5,\j) node{$2$};
}
\foreach \i in {1,3,5}{
  \draw[fill=white, draw=none] (\i,0.5) circle [radius=2.8mm];
  \draw (\i,0.5) node{$2$};
}
\node at (2.5,-1){$2$-$\Gamma_{4\times 6}$}; 
\end{tikzpicture}
\caption{A $2$-weighted grid graph.}\label{fig:2-weighted}
\end{figure}

\begin{thm}\label{thm3} Let $T_j(x)$ denote the $j$-th Chebyshev polynomial of
  the first kind, and let
  \[
  \zeta_{h,d}:=\cos\left(\frac{(2h-1)\pi}{4d}\right)
  \]
  for all integers $h$ and $d\neq 0$.
  Then for all integers $m,n\geq 1$, the following are equal:
\begin{enumerate}
  \item\label{thm3-1} the number of symmetric recurrents on $\sg_{(2m-1)\times(2n-1)}$;
  \item\label{thm3-2} the number of domino tilings of the $2$-weighted
    checkerboard of size $2m\times 2n$;
  \item\label{thm3-3} \ 
    \[
    (-1)^{mn}\,2^m\prod_{h=1}^m T_{2n}(i\,\zeta_{h,m});
    \]
  \item\label{thm3-4} \
    \[
    \prod_{h=1}^{m}\prod_{k=1}^{n}\left(4\,\zeta_{h,m}^2+4\,\zeta_{k,n}^2\right).
    \]
\end{enumerate}
\begin{remark}\label{remark:chebT2}
  As in Remark~\ref{remark:chebT}, we use
  identity~(\ref{eqn:half-angle}), this time to get
  \[
  T_{2n}(i\,\zeta_{h,m})=(-1)^n\,T_n(1+2\,\zeta_{h,m}^2),
  \]
  allowing us to equate the formula in part~(\ref{thm3-3}) with
  \[
  2^m\prod_{h=1}^mT_n(1+2\,\zeta_{h,m}^2).
  \]
  (We do not know of an analogous expression for the formula in
  Theorem~\ref{thm1}~(\ref{thm1-3}) in terms of products of $n$-th Chebyshev polynomials.)
\end{remark}
\end{thm}
\begin{proof}  The proof is similar to those for Theorem~\ref{thm1} and
  Theorem~\ref{thm2}.  Let $A_n'$ be the matrix defined at the beginning of the
  proof of Theorem~\ref{thm2}.  Then the symmetrized reduced Laplacian, $\tD^G$,
  for $2$-$\Gamma_{(2m-1)\times(2n-1)}$ is the matrix $D(m)$ displayed in the
  statement of Lemma~\ref{lemma:tridiagonal} after setting $A=B=A_n'$ and
  $C=2I_n$.  

\noindent [{\bf(\ref{thm2-1}) $=$ (\ref{thm2-2})}]:  The transpose $(\tD^G)^t$
is the reduced Laplacian of a sandpile graph, which we denote by $D''_{m\times
n}$ and embed in the plane as we did previously for $D_{m\times n}$ and
$D'_{m\times n}$ in Theorems~\ref{thm1} and~\ref{thm2}.  The embedding
of $D''_{m\times n}$ differs from that of $D'_{m\times n}$ only in
that each edge of the form $((m,i),(m-1,i))$ where $i\in [n]$ now
carries weight 2, again embedded as one edge coincident with the edge
$((m-1,i),(m,i))$ in the plane (Figure~\ref{fig:oddxodd 4x3} displays
$D''_{4\times 3}$).  
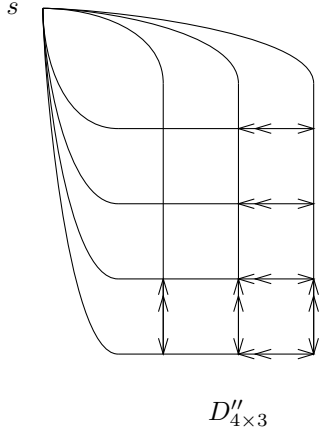
\begin{figure}[ht] 
\begin{tikzpicture}[scale=1]

\foreach \i in {0,1,2} {
\draw (\i,0) -- (\i,3.6);
\draw (\i,3.6) .. controls +(90:20pt) and +(0:20pt) .. (-1.60,4.6);
}
\foreach \j in {0,1,2,3} {
\draw (-0.60,\j) -- (1,\j);
\draw[<<->,>=mytip] (1,\j) -- (2,\j);
\draw (-0.60,\j) .. controls +(180:20pt) and +(270:20pt) .. (-1.60,4.6);
}
\foreach \i in {0,1,2}{
\draw[<->>,>=mytip] (\i,0) -- (\i,1);
}
\draw (-2,4.6) node(s){$s$};
\draw (1,-0.8) node(d){$D''_{4\times3}$};
\end{tikzpicture}
\caption{The symmetrized reduced Laplacian for $\sg_{7\times 5}$ is the reduced
Laplacian for $D''_{4\times 3}$. (The edge weights are encoded as in
Figure~\ref{fig:mobius 4x3}).}\label{fig:oddxodd 4x3}
\end{figure}
The result for this
section of the proof now follows just as it did in the proof of
Theorem~\ref{thm2}.

\noindent [{\bf(\ref{thm3-1}) $=$ (\ref{thm3-3})}]: By Corollary~\ref{cor:nsr}
and Lemma~\ref{lemma:tridiagonal}, the number of symmetric recurrents on
$2$-$\Gamma_{(2m-1)\times(2n-1)}$ is 
\[
\det\tD^G=(-1)^n\,\det(T),
\]
where
\[
T
=-A_n'\,U_{m-1}\left(\frac{A_n'}{2}\right)+2\,U_{m-2}\left(\frac{A_n'}{2}\right).
\]
Define
\[
s_{h,m}:=\cos\frac{(2h-1)\pi}{2m},\quad 1\leq h\leq m.
\]
Then, using identities from Section~\ref{subsection:tridiagonal},
\begin{align*}
  T&=-U_{m}\left(\frac{A_n'}{2}\right)+U_{m-2}\left(\frac{A_n'}{2}\right)\\[5pt]
  &=-2\,T_{m}\left(\frac{A_n'}{2}\right)\\
  &=-\prod_{h=1}^{m}(A_n'-2\,s_{h,m}I_n),
\end{align*}
Thus,
\[
\det\tD^G=\prod_{h=1}^{m}\chi_n(2\,s_{h,m}),
\]
where $\chi_n$ is the characteristic polynomial of $A_n'$.  Now consider the
recurrences~(\ref{thm2-chi}) and~(\ref{thm2-C}) in the proof of Theorem~\ref{thm2}.
Substituting $2s_{h,m}$ for $x$ in the former and $2-s_{h,m}$ for $x$ in the
latter, the two recurrences become the same.  It follows that 
$\chi_n(2\,s_{h,m})=2\,T_n(2-s_{h,m})$.  Then using a double-angle formula for cosine
and identity~(\ref{eqn:half-angle}),
\[
\chi_n(2\,s_{h,m})=2\,T_n(2-s_{h,m})=2\,T_n(1+2\,\zeta_{m-h+1,m}^2),
\]
and the result follows from Remark~\ref{remark:chebT2}.
\medskip

\noindent [{\bf(\ref{thm3-3}) $=$ (\ref{thm3-4})}]:  The numbers given in
parts~(\ref{thm2-3}) and~(\ref{thm2-4}) are equal by a straightforward
calculation, as in the proof of the analogous results in
Theorems~\ref{thm1} and \ref{thm2}.
\end{proof}
\begin{remark} Identities among trigonometric functions and among Chebyshev
  polynomials allow our formulae to be recast many ways.
  Remarks~\ref{remark:chebT},~\ref{remark:Lu-Wu}, and~\ref{remark:chebT2} have
  already provided some examples.  In addition, we note that in
  part~(\ref{thm2-4}) of Theorem~\ref{thm2} and in parts~(\ref{thm3-3})
  and~(\ref{thm3-4}) of Theorem~\ref{thm3}, one may replace each $\zeta_{h,n}$
  with $\sin((2h-1)\pi/(4n))$ or, as discussed at the end of the proof of
  Theorem~\ref{thm2}, with $\sin( (4h-1)\pi/(4n))$.
\end{remark}
\section{The order of the all-twos configuration}\label{section:order of all-2s}

Let $c$ be a configuration on a sandpile graph $\Gamma$, not necessarily an
element of $\sand(\Gamma)$, the sandpile group.  If $k$ is a nonnegative
integer, let $k\cdot c$ denote the vertex-wise addition of $c$ with itself $k$
times, without stabilizing.  The {\em order} of $c$, denoted $\mbox{order}(c)$,
is the smallest positive integer $k$ such that $k\cdot c$ is in the image of the
reduced Laplacian of $\Gamma$.  If $c$ is recurrent, then the order of $c$ is
the same as its order as an element of $\sand(\Gamma)$ according to the
isomorphism~(\ref{basic iso}).

Consider the sandpile grid graph, $\sg_{m\times n}$, with $m,n\geq 2$.  For each
nonnegative integer $k$, let $\vec{k}_{m\times n}=k\cdot\vec{1}_{m\times n}$ be
the {\em all-$k$s} configuration on $\sg_{m\times n}$ consisting of~$k$ grains
of sand on each vertex.  The motivating question for this section is: what is
the order of $\vec{1}_{m\times n}$?  Since $\vec{1}_{m\times n}$ has
up-down and left-right symmetry, its order must divide the order of the group of
symmetric recurrents on $\sg_{m\times n}$ calculated in
Theorems~\ref{thm1},~\ref{thm2}, and~\ref{thm3}.  The number of domino tilings
of a $2n\times 2n$ checkerboard can be written as $2^na_n^2$ where $a_n$ is an
odd integer (cf.~Proposition~\ref{prop:a_n}).  Our main result is
Theorem~\ref{thm4} which, through Corollary~\ref{cor:all-2s order}, says that the
order of $\vec{2}_{2n\times 2n}$ divides $a_n$.

\begin{prop}\label{prop:all-ones} Let $m,n\geq2$.
  \begin{enumerate}
  \item The configuration $\vec{1}_{m\times n}$ is not recurrent.
  \item The configuration $\vec{2}_{m\times n}$ is recurrent.
  \item\label{prop:all-ones3} The order of $\vec{1}_{m\times n}$ is either
    $\mathrm{order}(\vec{2}_{m\times n})$ or $2\,\mathrm{order}(\vec{2}_{m\times
    n})$.
  \item Let $\tD_{m\times n}$ be the reduced Laplacian of $\sg_{m\times n}$.
    The order of $\vec{1}_{m\times n}$ is the smallest integer $k$ such that
    $k\cdot\tD_{m\times n}^{-1}\vec{1}_{m\times n}$ is
    an integer vector.
  \end{enumerate}
\end{prop}
\begin{proof}
  Part~(1) follows immediately from the burning algorithm (Theorem~\ref{thm:Dhar}).
  For part~(2), we start by orienting some of the edges of $\sg_{m\times n}$ as shown in
  Figure~\ref{fig:orientation}.
\begin{figure}[ht] 
\begin{tikzpicture}[scale=0.7]
\SetVertexMath
\GraphInit[vstyle=Art]
\SetVertexNoLabel
\tikzset{VertexStyle/.style = {%
shape = circle,
shading = ball,
ball color = black,
inner sep = 1.5pt,
}}
\SetUpEdge[color=black]
\foreach \i in {1,2,3,4,5}{
  \foreach \j in {1,2,3,4}{
  \Vertex[x=\i,y=\j]{v\i\j}
  }
}
\Edge[style={->,>=mytip, semithick}](v14)(v13)
\Edge[style={->,>=mytip, semithick}](v24)(v23)
\Edge[style={->,>=mytip, semithick}](v34)(v33)
\Edge[style={->,>=mytip, semithick}](v44)(v43)
\Edge[style={->,>=mytip, semithick}](v54)(v53)

\Edge[style={->,>=mytip, semithick}](v13)(v12)
\Edge[style={->,>=mytip, semithick}](v23)(v22)
\Edge[style={->,>=mytip, semithick}](v33)(v32)
\Edge[style={->,>=mytip, semithick}](v43)(v42)
\Edge[style={->,>=mytip, semithick}](v53)(v52)

\Edge[](v12)(v11)
\Edge[](v22)(v21)
\Edge[](v32)(v31)
\Edge[](v42)(v41)
\Edge[](v52)(v51)

\Edge[style={->,>=mytip, semithick}](v11)(v21)
\Edge[style={->,>=mytip, semithick}](v12)(v22)
\Edge[style={->,>=mytip, semithick}](v13)(v23)
\Edge[style={->,>=mytip, semithick}](v14)(v24)

\Edge[style={->,>=mytip, semithick}](v21)(v31)
\Edge[style={->,>=mytip, semithick}](v22)(v32)
\Edge[style={->,>=mytip, semithick}](v23)(v33)
\Edge[style={->,>=mytip, semithick}](v24)(v34)

\Edge[style={->,>=mytip, semithick}](v31)(v41)
\Edge[style={->,>=mytip, semithick}](v32)(v42)
\Edge[style={->,>=mytip, semithick}](v33)(v43)
\Edge[style={->,>=mytip, semithick}](v34)(v44)

\Edge[](v41)(v51)
\Edge[](v42)(v52)
\Edge[](v43)(v53)
\Edge[](v44)(v54)

\tikzset{VertexStyle/.style = {inner sep = 0.0pt,}}
\foreach \i in {1,2,3,4,5}{
  \Vertex[x=\i,y=0]{s\i1}
  \Edge[style={->,>=mytip, semithick}](s\i1)(v\i1)
  \Vertex[x=\i,y=5]{s\i4}
  \Edge[style={->,>=mytip, semithick}](s\i4)(v\i4)
}
\foreach \j in {1,2,3,4}{
  \Vertex[x=0,y=\j]{s1\j}
  \Edge[style={->,>=mytip, semithick}](s1\j)(v1\j)
  \Vertex[x=6,y=\j]{s5\j}
  \Edge[style={->,>=mytip, semithick}](s5\j)(v5\j)
}
\end{tikzpicture}
\caption{Partial orientation of $\sg_{4\times5}$. Arrows pointing into the grid
from the outside represent edges from the sink vertex.}\label{fig:orientation}
\end{figure}
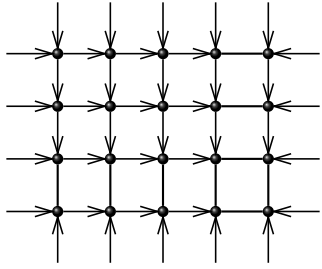
First, orient all the edges containing the sink, $s$, so that they point away
from~$s$.  Next, orient all the horizontal edges to point to the right except
for the last column of horizontal arrows.  Finally, orient all the vertical
edges down except for the last row of vertical arrows.  More formally, define
the {\em partial orientation} of $\sg_{m\times n}$,
\begin{align*}
\po:=
&\{( s,(i,j) ): 1\leq i\leq m,j\in\{1,n\} \}\\
&\cup\{( s,(i,j) ): i\in\{1,m\}, 1\leq j\leq n\}\\
&\cup\{( (i,j),(i,j+1) ):1\leq j\leq n-2\}\\
&\cup\{( (i,j),(i+1,j) ): 1\leq i\leq m-2\}.
\end{align*}

Use $\po$ to define a poset $P$ on the vertices of $\sg_{m\times n}$ by first setting
$u<_P v$ if $(u,v)\in\po$, then taking the transitive closure.  Now list the
vertices of $\sg_{m\times n}$ in any order $v_1,v_2,\dots$ such that $v_i<_Pv_j$
implies $i<j$.  Thus, $v_1=s$ and $v_2,v_3,v_4,v_5$ are the four corners of the
grid, in some order.  Starting from $\vec{2}_{m\times n}$, fire $v_1$.
This has the effect of adding the burning configuration to $\vec{2}_{m\times
n}$.  Since the indegree of each non-sink vertex with respect to $\po$ is $2$,
after $v_1,\dots,v_{i-1}$ have fired,~$v_i$ is unstable.  Thus, after firing the
sink, every vertex will fire while stabilizing the resulting configuration.  So
$\vec{2}_{m\times n}$ is recurrent by the burning algorithm.
\medskip

\noindent[{\sc note:} One way to think about listing the vertices, as prescribed
above, is as follows.  Let $P_{-1}:=\{s\}$, and for $i\geq 0$, let $P_i$ be
those elements whose distance from some corner vertex is~$i$.  (By {\em
distance} from a corner vertex, we mean the length of a longest chain in $P$ or
the length of any path in $\po$ starting from a corner vertex.)  For instance, $P_0$
consists of the four corners.  After firing the vertices in $P_{-1},P_0,\dots,P_{i-1}$,
all of the vertices in $P_i$ are unstable and can be fired in any order.]
\medskip

For part~(3), let $\alpha=\mathrm{order}(\vec{1}_{m\times n})$ and
$\beta=\mathrm{order}(\vec{2}_{m\times n})$, and let $e$ be the identity of
$\sand(\sg_{m\times n})$.  Let $\tL$ denote the image of the reduced Laplacian,
$\tD$, of $\sg_{m\times n}$.  Since $e=(2\alpha\cdot\vec{1}_{m\times
n})^{\circ}=(\alpha\cdot\vec{2}_{m\times n})^{\circ}$ and
$e=(\beta\cdot\vec{2}_{m\times n})^{\circ}=(2\beta\cdot\vec{1}_{m\times
n})^{\circ}$, we have
\begin{equation}\label{eqn:alpha beta}
  2\beta\geq\alpha\geq\beta.
\end{equation}

We have $(2\beta-\alpha)\cdot\vec{1}_{m\times n}=0\bmod\tL$.  Suppose
$\alpha\neq2\beta$.  It cannot be that $2\beta-\alpha=1$.  Otherwise,
$\vec{1}_{m\times n}=0\bmod\tL$.  It would then follow that $\vec{2}_{m\times
n}$ and $\vec{3}_{m\times n}$ are recurrent elements equivalent to $0$ modulo
$\tL$, whence, $\vec{2}_{m\times n}=\vec{3}_{m\times n}=e$, a contradiction.
Thus, $(2\beta-\alpha)\cdot \vec{1}_{m\times n}\geq\vec{2}_{m\times n}$.  Since
$\vec{2}_{m\times n}$ is recurrent, $((2\beta-\alpha)\cdot \vec{1}_{m\times
n})^{\circ}$ is recurrent and equivalent to $0$ modulo $\tL$, and thus must be
the~$e$.  So $2\beta-\alpha\geq\alpha$, and the right side of~\eqref{eqn:alpha beta} implies
$\alpha=\beta$, as required.

Now consider part~(4).  The order of $\vec{1}_{m\times n}$ is the smallest
positive integer $k$ such that $k\cdot\vec{1}_{m\times n}=0\bmod\tL$, i.e., for
which there exists an integer vector~$v$ such that $k\cdot\vec{1}_{m\times
n}=\tD_{m\times n}\,v$.  The result follows.
\end{proof}

\begin{example} We have $\mathrm{order}(\vec{1}_{2\times
  2})=2\,\mathrm{order}(\vec{2}_{2\times 2})=2$, and
  $\mathrm{order}(\vec{1}_{2\times 3})=\mathrm{order}(\vec{2}_{2\times 3})=7$.
  In general, we do not know which case will hold in part~\ref{prop:all-ones3}
  of Proposition~\ref{prop:all-ones}.
\end{example}

Table~\ref{table:all-2s} records the order of $\vec{2}_{m\times n}$ for
$m,n\in\{2,3,\dots,10\}$.  
\begin{table}[ht]
\centering
\begin{tabular}{c|lllllllll}
$m\backslash n$&2&3&4&5&6&7&8&9&10\\\hline
2&1&7&5&9&13&47&17&123&89\\
3&$\cdot$&8&71&679&769&3713&8449&81767&93127\\
4&$\cdot$&$\cdot$&3&77&281&4271&2245&8569&18061\\
5&$\cdot$&$\cdot$&$\cdot$&52&17753&726433&33507&24852386&20721019\\
6&$\cdot$&$\cdot$&$\cdot$&$\cdot$&29&434657&167089&265721&4213133\\
7&$\cdot$&$\cdot$&$\cdot$&$\cdot$&$\cdot$&272&46069729&8118481057&4974089647\\
8&$\cdot$&$\cdot$&$\cdot$&$\cdot$&$\cdot$&$\cdot$&901&190818387&1031151241\\
9&$\cdot$&$\cdot$&$\cdot$&$\cdot$&$\cdot$&$\cdot$&$\cdot$&73124&1234496016491\\
10&$\cdot$&$\cdot$&$\cdot$&$\cdot$&$\cdot$&$\cdot$&$\cdot$&$\cdot$&89893
\end{tabular}
\bigskip

\caption{Order of the all-2s element on $\sg_{m\times n}$ (symmetric in $m$
and $n$).}
\label{table:all-2s}
\end{table}
Perhaps the most striking feature of Table~\ref{table:all-2s} is the relatively
small size of the numbers along the diagonal ($m=n$).  It seems natural to
group these according to parity.  The sequence $\{\vec{2}_{2n\times
2n}\}_{n\geq1}$ starts $1,3,29,901,89893,\dots$, which is the beginning of the
famous sequence, $(a_n)_{n\geq1}$, we now describe.  The following was
established independently by several people (cf.~\cite{JSZ}):
\begin{prop}\label{prop:a_n}
  The number of domino tilings of a $2n\times 2n$ checkerboard has the form
  \[
   2^na_n^2
  \]
  where $a_n$ is an odd integer.
\end{prop}

For each positive integer $n$, let $P_n$ be the sandpile graph with vertices
\[
V(P_n)=\{v_{i,j}:1\leq i\leq n\mbox{ and }1\leq j \leq i\}\cup\{s\}.
\]
Each $v_{i,j}$ is connected to those vertices $v_{i',j'}$ such that
$|i-i'|+|j-j'|=1$.  In addition, every vertex of the form $v_{i,n}$ is connected
to the sink vertex, $s$.  The first few cases are illustrated in
Figure~\ref{fig:Pn}. 
\begin{figure}[ht] 
\begin{tikzpicture}[scale=0.8]
\SetVertexMath
\GraphInit[vstyle=Art]
\SetUpVertex[MinSize=3pt]
\SetVertexLabel
\tikzset{VertexStyle/.style = {%
shape = circle,
shading = ball,
ball color = black,
inner sep = 1.5pt
}}
\SetUpEdge[color=black]
\Vertex[NoLabel,x=0,y=0]{a11}
\Vertex[LabelOut=true,Lpos=0,L={s},x=1,y=0]{as}
\Edge[](a11)(as)
\Vertex[NoLabel,x=3,y=0]{b11}
\Vertex[NoLabel,x=4,y=0]{b12}
\Vertex[NoLabel,x=4,y=1]{b21}
\Vertex[LabelOut=true,Lpos=0,L={s},x=5,y=0.5]{bs}
\Edges(b11,b12,b21)
\Edge[](b12)(bs)
\Edge[](b21)(bs)
\Vertex[NoLabel,x=7,y=0]{c11}
\Vertex[NoLabel,x=8,y=0]{c12}
\Vertex[NoLabel,x=9,y=0]{c13}
\Vertex[NoLabel,x=8,y=1]{c22}
\Vertex[NoLabel,x=9,y=1]{c23}
\Vertex[NoLabel,x=9,y=2]{c33}
\Vertex[LabelOut=true,Lpos=0,L={s},x=10,y=1]{cs}
\Edges(c11,c12,c13,c23,c33)
\Edges(c12,c22,c23)
\Edge[](c13)(cs)
\Edge[](c23)(cs)
\Edge[](c33)(cs)
\draw (0.5,-0.8) node(G){$P_1$};
\draw (4,-0.8) node(G){$P_2$};
\draw (8.5,-0.8) node(G){$P_3$};
\end{tikzpicture}
\caption{}\label{fig:Pn}
\end{figure}
Next define a family of triangular checkerboards, $H_n$, as in
Figure~\ref{fig:Hn}.  The checkerboard $H_n$ for $n\geq 2$ is formed by adding
a $2\times (2n-1)$ array (width-by-height) of squares to the right of
$H_{n-1}$.
\begin{figure}[ht] 
\begin{tikzpicture}[scale=0.5]
\draw (0,0) -- (2,0) -- (2,1) -- (0,1) -- (0,0);
\draw (1,0) -- (1,1);
\draw (4,0) -- (8,0);
\draw (4,1) -- (8,1);
\draw (6,2) -- (8 ,2);
\draw (6,3) -- (8 ,3);
\draw (4,0) -- (4,1);
\draw (5,0) -- (5,1);
\draw (6,0) -- (6,3);
\draw (7,0) -- (7,3);
\draw (8,0) -- (8,3);
\draw (10,0) -- (16,0);
\draw (10,1) -- (16,1);
\draw (12,2) -- (16,2);
\draw (12,3) -- (16,3);
\draw (14,4) -- (16,4);
\draw (14,5) -- (16,5);
\draw (10,0) -- (10,1);
\draw (11,0) -- (11,1);
\draw (12,0) -- (12,3);
\draw (13,0) -- (13,3);
\draw (14,0) -- (14,5);
\draw (15,0) -- (15,5);
\draw (16,0) -- (16,5);
\draw (1,-0.8) node(G){$H_1$};
\draw (6,-0.8) node(G){$H_2$};
\draw (13,-0.8) node(G){$H_3$};
\end{tikzpicture}
\caption{}\label{fig:Hn}
\end{figure}
These graphs were introduced by M.~Ciucu~\cite{Ciucu} and later used by
L.~Pachter~\cite{Pachter} to give the first combinatorial proof of
Proposition~\ref{prop:a_n}.  As part of his proof, Pachter shows that~$a_n$
is the number of domino tilings of $H_n$.

As noted in~\cite{KPW}, considering $H_n$ as a planar graph and taking its dual
(forgetting about the unbounded face of $H_n$) gives the graph
$\mathcal{H}(P_n)$ corresponding to $P_n$ under the generalized Temperley
bijection of Section~\ref{section:Matchings and trees}.  See Figure~\ref{fig:Hn and H(Pn)}.
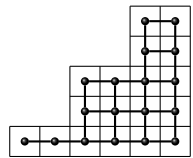
\begin{figure}[ht] 
\begin{tikzpicture}[scale=0.4]
\SetVertexMath
\GraphInit[vstyle=Art]
\SetUpVertex[MinSize=3pt]
\SetVertexLabel
\tikzset{VertexStyle/.style = {%
shape = circle,
shading = ball,
ball color = black,
inner sep = 1.1pt
}}
\SetUpEdge[color=black]
\Vertex[NoLabel,x=0.5,y=0.5]{a1}
\Vertex[NoLabel,x=1.5,y=0.5]{a2}
\Vertex[NoLabel,x=2.5,y=0.5]{a3}
\Vertex[NoLabel,x=3.5,y=0.5]{a4}
\Vertex[NoLabel,x=4.5,y=0.5]{a5}
\Vertex[NoLabel,x=5.5,y=0.5]{a6}
\Edges(a1,a2,a3,a4,a5,a6)
\Vertex[NoLabel,x=2.5,y=1.5]{b3}
\Vertex[NoLabel,x=3.5,y=1.5]{b4}
\Vertex[NoLabel,x=4.5,y=1.5]{b5}
\Vertex[NoLabel,x=5.5,y=1.5]{b6}
\Edges(b3,b4,b5,b6)
\Vertex[NoLabel,x=2.5,y=2.5]{c3}
\Vertex[NoLabel,x=3.5,y=2.5]{c4}
\Vertex[NoLabel,x=4.5,y=2.5]{c5}
\Vertex[NoLabel,x=5.5,y=2.5]{c6}
\Edges(c3,c4,c5,c6)
\Vertex[NoLabel,x=4.5,y=3.5]{d5}
\Vertex[NoLabel,x=5.5,y=3.5]{d6}
\Edges(d5,d6)
\Vertex[NoLabel,x=4.5,y=4.5]{e5}
\Vertex[NoLabel,x=5.5,y=4.5]{e6}
\Edges(e5,e6)
\Edges(a3,b3,c3)
\Edges(a4,b4,c4)
\Edges(a5,b5,c5,d5,e5)
\Edges(a6,b6,c6,d6,e6)
\draw (0,0) -- (6,0);
\draw (0,1) -- (6,1);
\draw (2,2) -- (6,2);
\draw (2,3) -- (6,3);
\draw (4,4) -- (6,4);
\draw (4,5) -- (6,5);
\draw (0,0) -- (0,1);
\draw (1,0) -- (1,1);
\draw (2,0) -- (2,3);
\draw (3,0) -- (3,3);
\draw (4,0) -- (4,5);
\draw (5,0) -- (5,5);
\draw (6,0) -- (6,5);
\end{tikzpicture}
\caption{$H_3$ and $\mathcal{H}(P_3)$.}\label{fig:Hn and H(Pn)}
\end{figure}
\begin{prop}\label{prop:an} The number of elements in the sandpile group for $P_n$ is
  \[
  \#\,\sand(P_n) =a_n,
  \]
  where $a_n$ is as in Proposition~\ref{prop:a_n}.
\end{prop}
\begin{proof}
  The number of domino tilings of $H_n$ equals the number of perfect matchings
  of $\mathcal{H}(P_n)$. By the generalized Temperley bijection, the latter
  is the number of spanning trees of $P_n$, and hence, the order of the sandpile group of
  $P_n$.  As mentioned above, Pachter shows in~\cite{Pachter} that~$a_n$ is the
  number of domino tilings of $H_n$.
\end{proof}

The main result of this section is the following:
\begin{thm}\label{thm4}
  Let $\langle\vec{2}_{2n\times 2n}\rangle$ be the cyclic subgroup of
  $\sand(\sg_{2n\times 2n})$ generated by the all-$2$s element of
  $\Gamma_{2n\times 2n}$, and let
  $\vec{2}_{n}$ denote the all-$2$s element on $P_n$.  Then the mapping
  \[
  \psi\colon\langle\vec{2}_{2n\times 2n}\rangle\to\sand(P_n),
  \]
  determined by $\psi(\vec{2}_{2n\times 2n})=\vec{2}_n$, is a well-defined
  injection of groups.
\end{thm}
\begin{proof}
Let $\tV_n$ and $\tV_{2n\times 2n}$ denote the non-sink vertices of $P_n$ and
$\sg_{2n\times 2n}$, respectively. We view configurations on $P_n$ as triangular
arrays of natural numbers and configurations on $\sg_{2n\times 2n}$ as
$2n\times 2n$ square arrays of natural numbers.  Divide the $2n\times 2n$ grid
by drawing bisecting horizontal, vertical, and diagonal lines, creating eight
wedges.  Define $\phi\colon\Z \tV_n\to \Z\tV_{2n\times 2n}$, by placing a
triangular array in the position of one of these wedges, then flipping about
lines, creating a configuration on $\sg_{2n\times 2n}$ with dihedral symmetry.
Figure~\ref{fig:phi} illustrates the case $n=4$.

\begin{figure}[ht] 
\begin{tikzpicture}[scale=0.4]
  \draw (-1,0) node(a){
  $
  \begin{array}{cccc}
    &&&j\\
    &&h&i\\
    &e&f&g\\
    a&b&c&d
  \end{array}
  $
  };

  \draw (14,0) node(b){
  $
  \begin{array}{cccccccc}
    j&i&g&d&d&g&i&j\\
    i&h&f&c&c&f&h&i\\
    g&f&e&b&b&e&f&g\\
    d&c&b&a&a&b&c&d\\
    d&c&b&a&a&b&c&d\\
    g&f&e&b&b&e&f&g\\
    i&h&f&c&c&f&h&i\\
    j&i&g&d&d&g&i&j
    \end{array}
  $
  };
  \draw [dashed] (14,-4.3) -- (14,4.3);
  \draw [dashed] (8.5,0) -- (19.5,0);
  \draw [dashed] (8.5,-4.3) -- (19.5,4.3);
  \draw [dashed] (8.5,4.3) -- (19.5,-4.3);
  \draw [fill=blue!20, opacity=0.2] (14,0) -- (19,4) -- (19,0) -- (14,0);
  \path[|->] (3,0) edge node[above]{$\phi$} (7,0);
\end{tikzpicture}
\caption{$\phi\colon\Z P_4\to\Z \sg_{8\times 8}$.}\label{fig:phi}
\end{figure}

We define special types of configurations on $P_n$.  First, let $s_n$ be
the configuration in which the number of grains of sand on each vertex records that
vertex's distance to the sink; then let $t_n$ denote the sandpile with no sand
except for one grain on each vertex along the boundary diagonal, i.e., those
vertices with degree less than~$3$.  Figure~\ref{fig:special} illustrates the
case $n=4$.

\begin{figure}[ht] 
\begin{tikzpicture}[scale=0.4]
  \draw (0,0) node(a){
  $
  \begin{array}{cccc}
    &&&1\\
    &&2&1\\
    &3&2&1\\
    4&3&2&1
  \end{array}
  $
  };

  \draw (9,0) node(b){
  $
  \begin{array}{cccc}
    &&&1\\
    &&1&0\\
    &1&0&0\\
    1&0&0&0
  \end{array}
  $
  };
\draw (0,-3) node{$s_4$};
\draw (9,-3) node{$t_4$};
\end{tikzpicture}
\caption{Special configurations on $P_4$.}\label{fig:special}
\end{figure}
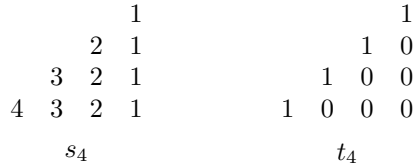

Let $\tD_{n}$ and $\tD_{2n\times 2n}$ be the reduced Laplacians for  $P_n$ and
$\sg_{2n\times 2n}$, respectively.  The following are straightforward
calculations:
\begin{enumerate}
  \item $\tD_n s_n=t_n$.
  \item If $c\in\Z P_n$, then $\tD_{2n\times 2n}(\phi(c))$ equals
    $\phi(\tD_n(c))$ at all non-sink vertices of $\sg_{2n\times 2n}$ except along the
diagonal and anti-diagonal, where they differ by a factor of~$2$:
    \[
    \tD_{2n\times 2n}(\phi(c))_{ij}=
    \begin{cases}
      2\,\phi(\tD_n(c))_{ij}&\text{for $i=j$ or $i+j=2n+1$,}\\
      \phi(\tD_n(c))_{ij}&\text{otherwise.}
    \end{cases}
    \]
\end{enumerate}
Let $\tL_n\subset\Z V_n$ and $\tL_{2n\times 2n}\subset\Z V_{2n\times 2n}$ denote
the images of $\tD_n$ and $\tD_{2n\times 2n}$, respectively.  Identify the
sandpile groups of $P_n$ and $\sg_{2n\times 2n}$ with $\Z V_n/\tL_n$ and $\Z
V_{2n\times 2n}/\tL_{2n\times 2n}$, respectively.  To show that $\psi$ is
well-defined and injective, we need to show that $k\,\vec{2}_n\in\tL_n$ for some
integer $k$ if and only if $k\,\vec{2}_{2n\times 2n}\in\tL_{2n\times 2n}$.
Since the reduced Laplacians are invertible over $\Q$, there exist unique
vectors $x$ and~$y$ defined over the rationals such that
\[
\tD_nx=\vec{2}_n\quad\text{and}\quad \tD_{2n\times 2n}y=\vec{2}_{2n\times 2n}.
\]
Using the special configurations $s_n$ and $t_n$ and the two calculations noted
above,
\[
\tD_nx=\vec{2}_n\quad\Longrightarrow\quad \tD_n(x-s_n)=\vec{2}_n-t_n
\quad\Longrightarrow\quad \tD_{2n\times 2n}\phi(x-s_n)=\vec{2}_{2n\times 2n}.
\]
In other words,   
\begin{equation}
  y=\phi(x-s_n).
  \label{eqn:y}
\end{equation}
Using the fact that $\tD_n$ is invertible over $\Q$, we see that
$k\,\vec{2}_n\in\tL_n$ if and only if $kx$ has integer coordinates.
By~(\ref{eqn:y}), this is the same as saying $ky$ has integer components, which
in turn is equivalent to $k\,\vec{2}_{2n\times 2n}\in\tL_{2n\times 2n}$, as required. 
\end{proof}
Combining this result with Proposition~\ref{prop:an} gives
\begin{cor}\label{cor:all-2s order}
  The order of $\vec{2}_{2n\times 2n}$ divides $a_n$. 
\end{cor}
\section{Conclusion}\label{section:conclusion}
We conclude with a list of suggestions for further work.
\medskip

\noindent{\bf 1.} Theorem~\ref{thm2} states that the number of domino tilings of
a M\"obius checkerboard equals the number of domino tilings of an associated
ordinary checkerboard after assigning weights to certain domino positions.  We
would like to see a direct bijective proof---one that does not rely on the Lu-Wu
formula (and thus giving a new proof of that formula).   For instance, consider
the tiling of the $4\times4$ checkerboard that appears second in the top row of
Figure~\ref{fig:checker4x4}.  It has one domino of weight~$2$.  So this weighted
tiling should correspond to two tilings of the $4\times4$ M\"obius checkerboard.
Presumably, one of these two tilings is just the unweighted version of the given
tiling.  One might imagine that the other tiling would result from pushing the
single blue domino to the right one square so that it now wraps around on the
M\"obius checkerboard, and then making room for this displacement by
systematically shifting the other dominos.
\smallskip

\noindent{\bf 2.} Section~\ref{section:order of all-2s} is motivated by Irena
Swanson's question: what is the order of the all-$1$s configuration,
$\vec{1}_{m\times n}$, on the $m\times n$ sandpile grid graph?
Proposition~\ref{prop:all-ones}~(\ref{prop:all-ones3}) shows this order is
either the same as or twice the order of the all-$2$s
configuration,~$\vec{2}_{m\times n}$.  It would be nice to know when each case
holds.  Corollary~\ref{cor:all-2s order} says the order of $\vec{2}_{2n\times
2n}$ divides the integer $a_n$ of Proposition~\ref{prop:a_n}, connected with
domino tilings.  When is this order equal to $a_n$?  Ultimately, of course, we
would like to know the answer to Swanson's original question.

\noindent{\bf 3.} Example~\ref{example:symm4x4} introduces an action of the
sandpile group of the $2m\times 2n$ sandpile grid graph on the domino tilings of
the $2m\times 2n$ checkerboard.  Perhaps this group action deserves further study.

\noindent{\bf 4.} To summarize some of the main ideas of this paper, suppose a
group acts on an arbitrary sandpile graph $\Gamma$.  If the corresponding
symmetrized reduced Laplacian or its transpose is the (ordinary) reduced
Laplacian of a sandpile graph~$\Gamma'$, then Proposition~\ref{prop:symmetric
subgroup iso} yields a group isomorphism between the symmetric configurations on
$\Gamma$ and the sandpile group $\sand(\Gamma')$ of $\Gamma'$.  By the
matrix-tree theorem, the size of the latter group is the number of spanning
trees of $\Gamma'$ (and, in fact, as mentioned earlier, $\sand{\Gamma'}$ is
well-known to act freely and transitively on the set of spanning trees of
$\Gamma'$).  The generalized Temperley bijection then gives a correspondence
between the spanning trees of $\Gamma'$ and perfect matchings of a corresponding
graph, $\mathcal{H}(\Gamma')$. Thus, the number of symmetric recurrents
on $\Gamma$ equals the number of perfect matchings of $\mathcal{H}(\Gamma')$.
We have applied this idea to the case of a particular group acting on sandpile
grid graphs.  Does it lead to anything interesting when applied to other classes
of graphs with group action?  The Bachelor's thesis of the first
author~\cite{Florescu} includes a discussion of the case of a dihedral action on
sandpile grid graphs.


\bibliography{tilings}{}
\bibliographystyle{amsplain}
\end{document}